\documentclass[11pt]{amsart}
\usepackage{amssymb,amsthm,amsmath,tikz,subfig}

\newtheorem{thm}{Theorem}
\newtheorem{cor}[thm]{Corollary}
\newtheorem{prop}[thm]{Proposition}
\newtheorem{lem}[thm]{Lemma}

\newtheorem{defn}[thm]{Definition}
\newtheorem{question}[thm]{Question}
\newtheorem{conjecture}[thm]{Conjecture}

\theoremstyle{definition}

\newtheorem{rem}[thm]{Remark}

\numberwithin{thm}{section}
\numberwithin{equation}{thm}

\newcommand{\A}{\ensuremath{{\mathbb{A}}}}
\newcommand{\bP}{\ensuremath{{\mathbb{P}}}}
\newcommand{\C}{\ensuremath{{\mathbb{C}}}}
\newcommand{\Z}{\ensuremath{{\mathbb{Z}}}}
\renewcommand{\P}{\ensuremath{{\mathbb{P}}}}
\newcommand{\Q}{\ensuremath{{\mathbb{Q}}}}
\newcommand{\bQ}{\ensuremath{{\mathbb{Q}}}}
\newcommand{\Qbar}{\ensuremath{\overline {\mathbb{Q}}}}

\newcommand{\N}{\ensuremath{{\mathbb{N}}}}
\newcommand{\p}{\ensuremath{{\mathfrak{p}}}}

\newcommand{\fp}{\ensuremath{{\mathfrak{p}}}}
\newcommand{\fq}{\ensuremath{{\mathfrak{q}}}}
\newcommand{\fm}{\ensuremath{{\mathfrak{m}}}}
\newcommand{\fo}{\ensuremath{{\mathfrak{o}}}}

\newcommand{\cZ}{\mathcal Z}
\newcommand{\cK}{\mathcal K}
\newcommand{\cL}{\mathcal L}
\newcommand{\cM}{\mathcal M}

\newcommand{\cY}{\mathcal Y}
\newcommand{\cX}{\mathcal X}
\newcommand{\cW}{\mathcal W}

\newcommand{\lra}{\longrightarrow}
\newcommand{\Kbar}{\ensuremath {\overline K}}

\newcommand{\ba}{{\boldsymbol \alpha}}

\renewcommand{\O}{\ensuremath{{\mathcal{O}}}}

\DeclareMathOperator{\Gal}{Gal}

\DeclareMathOperator{\Aut}{Aut}

\begin{document}

\title[Iterated Galois groups of cubic polynomials]{Finite index theorems for iterated Galois groups of cubic polynomials}


%    author one information
\author[A. Bridy]{Andrew Bridy}
\address{Andrew Bridy\\Department of Mathematics\\ Texas A\&M University\\
College Station, TX, 77843, USA}
\email{andrewbridy@tamu.edu}

%    author two information
\author[T. J. Tucker]{Thomas J. Tucker}
\address{Thomas J. Tucker\\Department of Mathematics\\ University of Rochester\\
Rochester, NY, 14620, USA}
\email{thomas.tucker@rochester.edu}

%    \subjclass is required.


\date{}

\dedicatory{}

%    Abstract is required.
\begin{abstract}
  Let $K$ be a number field or a function field. Let $f\in K(x)$ be a
  rational function of degree $d\geq 2$, and let
  $\beta\in \P^1(\Kbar)$. For all $n\in\mathbb{N}\cup\{\infty\}$, the
  Galois groups $G_n(\beta)=\Gal(K(f^{-n}(\beta))/K(\beta))$ embed
  into $\Aut(T_n)$, the automorphism group of the $d$-ary rooted tree
  of level $n$. A major problem in arithmetic dynamics is the arboreal
  finite index problem: determining when
  $[\Aut(T_\infty):G_\infty(\beta)]<\infty$.  When $f$ is a cubic
  polynomial and $K$ is a function field of transcendence degree $1$
  over an algebraic extension of $\Q$, we resolve this problem by
  proving a list of necessary and sufficient conditions for finite
  index. This is the first result that gives necessary and sufficient
  conditions for finite index, and can be seen as a dynamical analog
  of the Serre Open Image Theorem. When $K$ is a number field, our
  proof is conditional on both the $abc$ conjecture for $K$ and
  Vojta's conjecture for blowups of $\P^1 \times\P^1$. We also use our
  approach to solve some natural variants of the finite index problem
  for modified trees.
\end{abstract}

\maketitle


\section{Introduction and Statement of Results}
Let $K$ be a field.  Let $f\in K(x)$ with $d=\deg f\geq 2$ and let
$\beta\in \P^1(\Kbar)$.  For $n\in\N$, let
$K_n(f,\beta)=K(f^{-n}(\beta))$ be the field obtained by adjoining the
$n$th preimages of $\beta$ under $f$ to $K(\beta)$ (we declare that
$K(\infty)=K$).  Also set
$K_\infty(f,\beta)=\bigcup_{n=1}^\infty K_n(f,\beta)$.  For
$n\in\N\cup\{\infty\}$, define
$G_n(f,\beta)=\Gal(K_n(f,\beta)/K(\beta))$.  In most of the paper, we
will write $G_n(\beta)$ and $K_n(\beta)$, suppressing the dependence
on $f$ if there is no ambiguity.

The group $G_\infty(\beta)$ embeds into $\Aut(T_\infty)$, the automorphism 
group of an infinite $d$-ary rooted tree $T_\infty$. Recently there has been much interest in the problem 
of determining when the index $[\Aut(T_\infty):G_\infty(\beta)]$ is finite. See Section \ref{Background} for 
background on this problem and previous results. The group $G_\infty(\beta)$ is the image of an 
arboreal Galois representation, so this finite index problem is a natural analog in arithmetic dynamics of 
the finite index problem for the $\ell$-adic Galois representations associated to elliptic curves, 
resolved by Serre's celebrated Open Image Theorem~\cite{Serre}.

We resolve the finite index problem for cubic polynomials when $K$ is a function field 
of transcendence degree one over an algebraic extension of $\Q$. For such function fields, we  
prove an explicit list of necessary and sufficient conditions for finite index. 
When $K$ is a number field, our proof is conditional on both the $abc$ conjecture and Vojta's 
conjecture for blowups of $\P^1 \times\P^1$. 

To explain our main theorem, we recall some standard terminology from arithmetic dynamics. 
Let $f\in K[x]$. For $n\geq 1$, let $f^n$ denote $f$ composed with itself $n$ times. 
Let $f^0=x$, the compositional identity.  
For $\alpha\in\P^1(\Kbar)$, the \emph{forward orbit} of $\alpha$ under $f$ 
is the set $\mathcal{O}_f(\alpha)=\{f^n(\alpha):n\geq 1\}$. We say $\alpha$ is \emph{periodic} 
for $f$ if $f^n(\alpha)=\alpha$ for some $n\geq 1$ and \emph{preperiodic} for $f$ if 
$f^n(\alpha)=f^m(\alpha)$ for some $n>m\geq 1$. 
The point $\beta\in \P^1(\Kbar)$ is \emph{postcritical} for $f$ if there 
is some critical point (ramification point) $\gamma$ of $f$ such that $\beta\in\mathcal{O}_f(\gamma)$. 
We say $f$ is \emph{postcritically finite} or \emph{PCF} if only finitely many 
points in $\P^1(\Kbar)$ are postcritical for $f$. 
For $\beta\in \Kbar$, the pair $(f,\beta)$ is \emph{stable} if $f^n(x)-\beta$ is irreducible over $K(\beta)$ 
for all $n\geq 1$ and \emph{eventually stable} if the number of irreducible factors of 
$f^n(x)-\beta$ over $K(\beta)$ is bounded independently of $n$ as $n\to\infty$ 
(stability and eventual stability can 
also be defined for rational functions as in~\cite{RafeAlonEventualStability}).
If $K$ is a function field with field of constants $k$, 
then $f$ is \emph{isotrivial} if there exists a degree $1$ polynomial $\sigma\in \Kbar(x)$ such 
that $\sigma\circ f\circ \sigma^{-1}\in k[x]$.

Our main theorem is as follows.

\begin{thm}\label{thm: main}
  Let $K$ be a number field or a function field of transcendence
  degree one over an algebraic extension of $\Q$.  Let $f \in K[x]$ with
  $\deg f=3$ and let $\beta\in K$. If $K$ is a number field, assume
  the $abc$ conjecture for $K$ and Vojta's conjecture for blowups of
   $\P^1 \times\P^1$. If $K$ is a function
  field, assume that $f$ is not isotrivial.
  
  The following are equivalent:
  \begin{enumerate}
  \item The pair $(f,\beta)$ is eventually stable, 
  $\beta$ is not postcritical for $f$,
   $f$ is not PCF, 
   $f$ has distinct finite critical points $\gamma_1, \gamma_2$, and 
   $f^n(\gamma_1)\neq f^n(\gamma_2)$ for all $n\geq 1$.
   \item The group $G_\infty(\beta)$ has finite index in $\Aut(T_\infty)$.
  \end{enumerate}
\end{thm}

It is fairly straightforward to show that the conditions in Theorem~\ref{thm: main} are necessary. 
The proof of their necessity (Proposition~\ref{obstructions}) 
is unconditional for both number fields and function fields,  
and much of the argument does not 
depend on the degree of $f$ or whether $f$ is a polynomial. 
Thus we view the failure of each condition as a natural 
obstruction to a finite index result. The majority of our paper is dedicated 
to the much harder problem of showing that, in the case of cubic polynomials, these conditions 
are also sufficient -- that is, that these are the only obstructions
to finite index. It may seem that the eventual stability condition is very strong, 
but eventual stability is known for many explicit families of $f$ and $\beta$, and is conjectured to 
hold whenever $f$ is not isotrivial and $\beta$ is not periodic for $f$ (see~\cite{RafeAlonEventualStability}).

A function field $K$ of transcendence degree 1 over a field $k$ here
means, as usual, a finitely generated field extension of $k$ of
transcendence degree 1 such that $k$ is algebraically closed in $K$.
For our theorem, we may as well assume that $K$ is a function field of
transcendence degree 1 over $\Qbar$ since $\Gal(K_n(\beta) \cdot
\Qbar /K(\beta) \cdot \Qbar)$ is always a subgroup of
$\Gal(K_n(\beta)/K(\beta))$.  

We also note that in the function field case, in our proof of Theorem
\ref{thm: main}, it does not seem to be necessary that $K$ have
transcendence degree 1 over an algebraic extension of $\Q$.  The only place where that
assumption is used is in Proposition \ref{FunctionFinite} and that
result extends to function fields of transcendence dimension 1 over
any field of characteristic 0 (see Remark \ref{constant-field}).
However, since no proof of Proposition \ref{FunctionFinite} in this
more general setting exists in the literature, we state our result
only for function fields whose field of constants is an algebraic
extension of $\bQ$.

The main idea of our proof is to produce primes with specified ramification behavior 
in the tower of extensions $K_n(\beta)$ and translate this into information about the Galois groups $G_n(\beta)$. 
The basic strategy can be seen as an extension of our previous work on a 
Zsigmondy principle for ramification~\cite{BridyTucker} (see also
\cite{BIJJ, GNT, GNT2}), which 
uses the Call-Silverman canonical height for dynamical systems~\cite{CallSilverman} and 
the ``Roth-$abc$'' estimate as in~\cite{GranvilleSquarefrees} to show 
that a new prime ramifies in each $K_n(\beta)$ for large $n$. To extend our argument, 
we need stronger restrictions on ramification 
derived from sharper height estimates. In the number field case, 
these estimates come from work of Huang \cite{Keping} on the ``dynamical gcd" problem conditional on 
Vojta's conjecture, Xie's \cite{XieDMLA2} proof of 
the dynamical Mordell-Lang conjecture for $\A^2(\Qbar)$, the Medvedev-Scanlon \cite{MedvedevScanlon}
classification of varieties invariant under split polynomial mappings, and 
old work of Ritt \cite{Ritt1} on polynomials that commute under composition. 
In the function field case, we derive the required estimates from 
a dynamical Andr\'e-Oort theorem of Ghioca-Ye~\cite{GY} and
Favre-Gauthier \cite{Gauthier}, together with the 
Call-Silverman machinery of canonical height and specialization.

We also use our approach to solve two other finite index problems that do not seem to have 
been considered before. In Section~\ref{The stunted tree}, we replace the infinite $d$-ary tree $T_\infty$ 
with the actual tree of preimages of $\beta$ under $f$, which we call 
the ``stunted tree" $T^s_\infty(\beta)$. The idea is 
that in order to identify $T_\infty$ with a tree of preimages of $f$, 
we need to count solutions to $f^n(x)=\beta$ with 
multiplicity, but if $\beta$ is periodic or postcritical for $f$, 
this destroys any hope of a finite index result in $\Aut(T_\infty)$ 
for a somewhat coarse reason. We show that passing to the stunted tree 
(which is a quotient tree of $T_\infty$) recovers 
a finite index result in many cases. With a modified notion of eventual stability, 
we prove Theorem~\ref{finite index in stunted tree}, 
which gives a list of necessary and sufficient conditions for 
finite index in $\Aut(T^s_\infty(\beta))$ (conditional on $abc$ and
the Vojta conjecture in the case of number fields) 
that holds for all cubic polynomials except one exceptional family: polynomials 
that commute with a nontrivial M\"{o}bius transformation, which are well known 
to be problematic in this area. Note 
that this exceptional family does not need to be excluded in Theorem~\ref{thm: main}. In Section~\ref{The multitree}, 
we simultaneously consider multiple preimage trees and prove 
Theorem~\ref{multitree theorem}, a similar result in this context.

In the next section of the paper, we provide some background and
definitions.  Following that, in Section \ref{red ob}, we prove one
direction of Theorem \ref{thm: main} and perform some reductions.
Next, we discuss some of the issues around the notion of eventual
stability in Section \ref{stable}.  Then, in Section \ref{heights}, we
summarize the height inequalities that form the backbone of the main
arguments used in this paper; the inequalities here are similar those
used in \cite{BridyTucker, GNT, GNT2}.  Section \ref{Galois section}
contains some arguments from Galois theory, the most important of
which (Proposition \ref{final Galois}) shows that the disjointness of
certain field extensions can be derived from appropriate ramification
data.  In the next two sections, Sections \ref{number-case} and
\ref{function-case}, we prove the ``dynamical gcd'' results that allow
us to deal with the fact that our polynomial has two distinct critical
points; these results are what allow us to go beyond the previously
studied case of unicritical polynomials.  Section \ref{main proof}
contains the proof of Theorem \ref{thm: main}; we first prove the
existence of primes with certain ramification properties, in
Proposition \ref{final prop}, and then apply Proposition \ref{final
  Galois}.  As described above, Sections \ref{The stunted tree} and
\ref{The multitree} contain modifications of the tree of inverse
images which allow for more general finite index results. Finally, in
Section \ref{iso case}, we treat the case of isotrivial polynomials
over function fields.

Many of the techniques used in this paper should generalize to other
situations.  For example, it should certainly be possible to prove a
more general finite index theorem for all non-PCF cubic polynomials,
assuming eventual stability over function fields and assuming eventual
stability, Vojta's conjecture for blow-ups of $\bP^1 \times \bP^1$ and
the $abc$-conjecture for number fields (here, by ``finite index'', we
mean finite index in a natural group that may be smaller than
$\Aut(T_\infty)$).   Higher degree polynomials present additional
complications, since our arguments here use the fact that any
transitive subgroup of $S_3$ containing a transposition must be all of
$S_3$.  On the the other hand, it should be possible to use the
techniques here to treat iterated Galois groups of polynomials of
prime degree, since any subgroup of $S_p$ that contains a transposition and
$p$-cycle must be all of $S_p$.  Some new results along the lines of
\cite{Gauthier, GY} would be necessary to make the proof unconditional
(assuming eventual stability) over function fields.  Over number
fields, the required dynamical gcd results are still implied by Vojta's
conjecture for blow-ups of $\bP^1 \times \bP^1$, by work of \cite{Keping}.  Recently, Juul,
Krieger, Looper, Manes, Thompson, and Walton independently discovered
a similar proof of the number field case of Theorem \ref{thm: main}; again,
assuming $abc$ and Vojta's conjecture.  Their argument also uses
Huang's \cite{Keping} result on Vojta's conjecture to control
dynamical gcds along with methods similar to those of
\cite{BridyTucker, GNT} to obtain primitive prime divisors in the
orbits of one of the critical points.  They are working on extensions
to some higher degree cases.

Finally, note the similarity of Theorem~\ref{thm: main} to Conjecture 3.11 
in~\cite{RafeArborealSurvey} for degree $2$ rational functions. The main 
difference is that a degree $2$ rational function $f$ may commute with an 
nontrivial M\"obius transformation $\sigma$ 
such that $\sigma(\beta)=\beta$ without forcing $\beta$ to be postcritical for $f$ 
or $(f,\beta)$ to not be eventually stable. See~\cite{RafeArborealSurvey} for details.

\noindent {\bf Acknowledgements.}  We would like to thank Thomas
Gauthier, Dragos Ghioca, Keping Huang, Rafe Jones, Nicole Looper, and
Khoa Nguyen for many helpful conversations.  

\section{Background}\label{Background}

We give a brief introduction to arboreal Galois theory and 
describe some of the previous results in the area. For a comprehensive survey, see~\cite{RafeArborealSurvey}. 
Since we will be making repeated use of wreath products, we clarify our notation: if 
$G,H$ are groups and $H$ acts on a set $X$, then 
\[
G\wr_X H = G^{X} \rtimes H,
\] 
where $G^{X}$ is the direct product of $|X|$ copies of $G$ 
and $H$ acts on $G^{X}$ by permuting coordinates.

Let $d\geq 2$ be a fixed integer. Let $T_n$ denote the complete $d$-ary rooted tree of level $n$, 
where each non-leaf vertex has precisely $d$ child vertices. Let $T_\infty$ denote the complete infinite 
$d$-ary rooted tree, which is the direct limit of all $T_n$ as $n\to\infty$. 

\begin{figure}[ht]
\includegraphics[scale=.75]{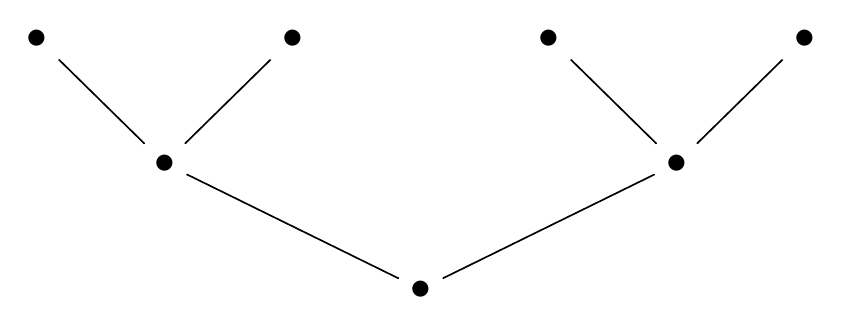}
\caption{$T_2$ for $d=2$}
\label{Binary T_2}
\end{figure}

Suppose $f\in K(x)$ with $\deg f=d$ and $\beta\in \P^1(\Kbar)$.
The morphism $f:\P^1(\Kbar)\to\P^1(\Kbar)$ is 
$d$-to-$1$ away from its critical points, so the set $f^{-n}(\beta)$ generically consists of $d^n$ points. Let
\[
T_n(\beta) = \bigsqcup_{i=0}^n f^{-i}(\beta).
\]
For a generic choice of $\beta$, one sees that $T_n(\beta)$ and $T_n$
are isomorphic.  The root is $\beta$, the vertices at the $i$th level
are labeled by the elements of $f^{-i}(\beta)$, and the edges connect
vertices labeled by $z$ and $f(z)$ for all $z\in f^{-i}(\beta)$ with
$1\leq i\leq n$.  However, if $\gamma\in f^{-m}(\beta)$ is a critical
point of $f$, then $f(\gamma)$ has fewer than $d$ preimages and
$|f^{-n}(\beta)|<d^n$ for $n\geq m$. In this case, we can repair the
construction as follows: we take $f^{-n}(\beta)$ to be the multiset of
solutions to $f^n(x)=\beta$, each counted with multiplicity.  Then
$|f^{-n}(\beta)|=d^n$ for all $n$, and once again we have
$T_n(\beta)\cong T_n$; see Figure~\ref{T_2 examples} for examples.
Note that it may be possible that some $z\in\P^1(\Kbar)$ appears in
$f^{-n}(\beta)$ for two different values of $n$, which happens when
$\beta$ is periodic. This potential repetition is why $T_n(\beta)$ is defined as a disjoint union, 
as taking a non-disjoint union would make $T_n(\beta)$ and $T_n$ 
non-isomorphic as trees in the case of periodic $\beta$. 
In Section \ref{The stunted tree} we define and study a ``stunted tree" 
that does not keep track of repeated points in the tree due to periodic or 
postcritical $\beta$, with somewhat different results. Let
\[
T_\infty(\beta) =\bigsqcup_{i=0}^\infty f^{-i}(\beta)
\]
be the infinite $d$-ary rooted tree formed by taking the direct limit 
of all the $T_n(\beta)$. With our convention on counting with multiplicity, $T_\infty(\beta)$ and $T_\infty$ 
are isomorphic as rooted trees for every $\beta\in \P^1(\Kbar)$.

\begin{figure}[ht]
\centering
\subfloat{\includegraphics[scale=.65]{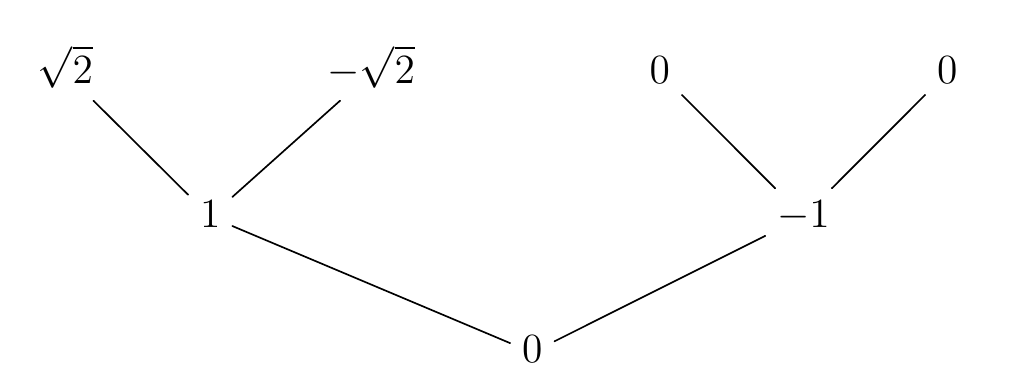}}
\subfloat{\includegraphics[scale=.65]{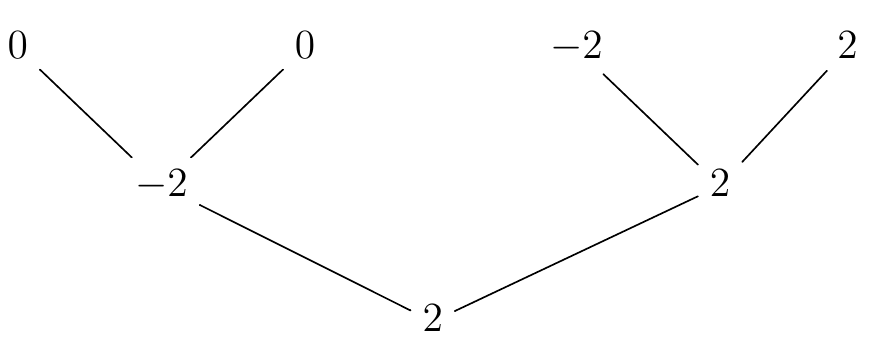}}
\caption{$T_2(0)$ for $x^2-1$ and $T_2(2)$ for $x^2-2$}
\label{T_2 examples}
\end{figure}

As in the introduction, for $n\in\mathbb{N}\cup\{\infty\}$, set $K_n(\beta)=K(f^{-n}(\beta))$ and 
$G_n(\beta)=\Gal(K_n(\beta)/K(\beta))$. 
The fields $K_n(\beta)$ are obviously Galois extensions of $K(\beta)$.  
 Note that $G_\infty(\beta)$ is isomorphic to the inverse 
limit of the groups $G_n(\beta)$ with the natural projection maps $G_{n+1}(\beta)\to G_n(\beta)$. 

The starting point of arboreal Galois theory is the observation that, 
for all $n\geq 1$, the group $G_n(\beta)$
acts faithfully by automorphisms on $T_n(\beta)$.  Edges in
$T_n(\beta)$ are determined by the action of $f$ on
$T_n(\beta)\setminus \beta$, and the action of $G_n(\beta)$ on
$T_n(\beta)$ commutes with the action of $f$ because $f\in K(x)$. Thus
there is an injection
\[G_n(\beta)\hookrightarrow \Aut(T_n(\beta))\cong\Aut(T_n),\]
where $\Aut(T_n)$ is the automorphism group of $T_n$ as a rooted tree. It is also clear that  
$\Gal(K_n(\beta)/K_{n-1}(\beta))$ injects into $\Aut(T_n/T_{n-1})$, the subgroup of $\Aut(T_n)$ that fixes 
the $(n-1)$st level of the tree.

Certainly, $\Aut(T_1)\cong S_d$. It is easy to show inductively that
\[
\Aut(T_n) \cong S_d\wr_{T_{n-1}\setminus T_{n-2}} \Aut(T_{n-1}),
\]
from which follows the order formula 
\[
|\Aut(T_n)| = (d!)^{d^{n-1}}|\Aut(T_{n-1})|= (d!)^{d^{n-1}+d^{n-2}+\dots + 1}=(d!)^{\frac{d^n-1}{d-1}}.
\]
Also observe that $\Aut(T_n/T_{n-1})\cong (S_d)^{|T_{n-1}\setminus T_{n-2}|}$, so 
\[
|\Aut(T_n/T_{n-1})|=(d!)^{d^{n-1}}.
\]
The group $\Aut(T_n)$ is often described as an ``iterated wreath product''
\[
\Aut(T_n)\cong S_d\wr S_d\wr S_d\wr\dots\wr S_d,
\]
though this is somewhat misleading, as every instance of the symbol $\wr$ in the above line 
implicitly refers to a group action on a distinct set.

Taking inverse limits, there is an injection $G_\infty(\beta)\hookrightarrow\Aut(T_\infty)$. Equivalently, 
we have the \emph{arboreal Galois representation} 
\[
\rho_{f,\beta}:\Gal(\Kbar/K)\to \Aut(T_\infty),
\]
where $\rho_{f,\beta}$ sends an element of $\Gal(\Kbar/K)$ to the induced automorphism of $T_\infty$. Then $G_\infty(\beta)$ is the image of $\rho_{f,\beta}$ inside $\Aut(T_\infty)$. 
The main problem in the field is to determine the ``size" of $G_\infty(\beta)$, for 
example as in the following motivating question.
\begin{question}\label{main question}
Let $K$, $f$, and $\beta$ be as above.
\begin{enumerate}
\item When is $[\Aut(T_\infty):G_\infty(\beta)]=1$?
\item When is $[\Aut(T_\infty):G_\infty(\beta)]<\infty$?
\end{enumerate}
\end{question}

The first part of Question~\ref{main question} was originally studied by 
Odoni~\cite{OdoniIterates, OdoniWreathProducts}, who showed 
that $[\Aut(T_\infty):G_\infty(\beta)]=1$ (i.e. $\rho_{f,\beta}$ is surjective) 
when $f$ is a polynomial with generic (transcendental) coefficients. 
More recently, Juul showed that $G_\infty(\beta)=\Aut(T_\infty)$ for generic rational functions~\cite{Juul}. 
Despite the work of Odoni and Juul, it is hard to show that $G_\infty(\beta)=\Aut(T_\infty)$ 
for any particular choice of $f$ and $\beta$. Still unresolved is Odoni's conjecture~\cite[Conjecture 7.5]{OdoniIterates} 
that for every $d\geq 2$, 
there exists a degree $d$ polynomial with $G_\infty(\beta)=\Aut(T_\infty)$ 
(though recent work of Looper~\cite{Looper} gives such a polynomial in every prime degree). 

The second part of Question~\ref{main question} is perhaps more natural, because the 
answer is invariant under replacing $K$ by a finite extension (see Proposition~\ref{reductions}). 
The arboreal finite index problem 
is a natural analog in arithmetic dynamics of the finite index problem for the 
$\ell$-adic Galois representations associated to elliptic curves, solved by 
Serre's Open Image Theorem~\cite{Serre}. Following Serre's theorem,
some obstructions to 
finite index can be thought of as dynamical analogs of complex multiplication.

We comment briefly on the general strategy for showing that $G_\infty(\beta)$ has finite or infinite 
index in $\Aut(T_\infty)$. Consider the natural projection map $\phi_n:\Aut(T_\infty)\to T_n$. 
The restriction of $\phi_n$ to $G_\infty(\beta)$ maps onto $G_n(\beta)$. By 
elementary group theory, $[\Aut(T_\infty):G_\infty(\beta)]\geq [\Aut(T_n):G_n(\beta)]$. 
So if $[\Aut(T_n):G_n(\beta)]\to\infty$ as $n\to\infty$, or equivalently if 
\[
[\Aut(T_n/T_{n-1}): \Gal(K_n(\beta)/K_{n-1}(\beta))]>1
\]
for infinitely many $n$, 
then $[\Aut(T_\infty):G_\infty(\beta)]=\infty$. 
On the other hand, the profinite structure of $G_\infty(\beta)$ implies that 
distinct cosets of $G_\infty(\beta)$ in $T_\infty$ 
must map under $\phi_n$ to distinct cosets of $G_n(\beta)$ in $T_n$ for some $n$. 
So if $[\Aut(T_n):G_n(\beta)]$ remains bounded as $n\to\infty$, or equivalently if 
\[
\Gal(K_n(\beta)/K_{n-1}(\beta))=\Aut(T_n/T_{n-1})\cong (S_d)^{d^{n-1}}
\]
for all large $n$, then $[\Aut(T_\infty):G_\infty(\beta)]<\infty$.

For more on arboreal Galois representations and related topics, see for  example~\cite{Cremona,Stoll,BostonJonesArboreal,BostonJonesImage,JonesComp,JonesDensity,
JonesFixedPointFree,JonesManes,
Hindes1,Hindes2,HamblenJonesMadhu,Krieger,IngramUniformization,
IngramSilverman,BIJJ,PinkFinitenessLiftability,
PinkQuadraticInfiniteOrbits,PinkQuadratic}. Most of the previous work is for polynomials 
and rational functions of degree 2. The following are typical theorems.
\begin{thm}\cite{JonesDensity}
Let $K=\Q$. Let $f\in \Z[x]$ be monic with $\deg f=2$. Suppose that $f$ is not PCF, $(f,0)$ 
is stable, and $0$ is strictly preperiodic for $f$. Then $[\Aut(T_\infty):G_\infty(0)]<\infty$.
\end{thm}
\begin{thm}\cite{GNT}\label{from GNT}
Let $K=\Q$. Let $f\in \Z[x]$ be monic with $\deg f=2$. Suppose that $f$ is not PCF and that 
$(f,0)$ is stable. Assume the $abc$ conjecture for $\Q$. Then $[\Aut(T_\infty):G_\infty(0)]<\infty$.
\end{thm}

Our current paper deals only with the case of cubic polynomials, but
we will treat the case of quadratic polynomials in future work.
It is not difficult to prove analogs of the main theorems of this
paper for quadratics using our techniques here.  With a bit more work,
and an extension of \cite{UnicriticalAndreOort}, we obtain,
in \cite{BGHT}, a stronger variant of Theorem \ref{multitree theorem} that
applies to families of inverses of images of points under several
different quadratic polynomials.  These results are unconditional for
quadratic polynomials over function fields, and require $abc$ and
Vojta's conjecture for number fields.  We note that in \cite{DLT} it
is shown that that if $f$ is a nonisotrivial quadratic polynomial over
a function field and $\beta$ is not periodic, then $(f,\beta)$ is
eventually stable.

\begin{rem}
It should be noted that index is not the only measure of the size of a subgroup relative to $\Aut(T_\infty)$. 
Hausdorff dimension is a refined measurement of the size of an infinite index subgroup  
linked to interesting dynamical and arithmetic properties. 
See~\cite{BostonJonesArboreal,BostonJonesImage,BFHJY}.
\end{rem}

\section{Reductions and Obstructions}\label{red ob} 

In this section we prove some reductions that will be used
in the proof of Theorem~\ref{thm: main}. For many of our arguments, it
will be convenient to make a change of variables so that our cubic
polynomial $f$ is monic and the quadratic term vanishes. This may
require taking a finite extension of the ground field $K$.  In
Proposition~\ref{reductions}, we show that neither of these
modifications affects the answer to the finite index question.

We begin with the following standard lemma from Galois theory, which
we state without proof. 

\begin{lem}\label{degrees}
Let $M_1 \subset M_2$ be a field extension of finite degree and let $M_1
\subseteq M_3$ be any field extension.  Then we have
\begin{equation}\label{any}
[M_2 \cdot M_3: M_3] \leq [M_2 : M_1].
\end{equation}
If $[M_3: M_1]$ is finite, we also have
\begin{equation} \label{finite}
[M_2 \cdot M_3: M_3] \geq \frac{[M_2:M_1]}{[M_3:M_1]}.
\end{equation}
\end{lem}

\begin{prop}\label{reductions}
  Let $f\in K(x)$ with $\deg f\geq 2$ and let $\beta\in \P^1(\Kbar)$. Let
  $K'$ be any extension of $K$.  For $n\in\mathbb{N}\cup\{\infty\}$,
  let $K'_n(f,\beta)=K'(f^{-n}(\beta))$ and
  $G'_n(f,\beta)=\Gal(K'_n(f,\beta)/K'(\beta))$.  Let $\sigma\in K(x)$
  have degree one, and let $g=\sigma\circ f\circ\sigma^{-1}$.  Then
  $[\Aut(T_\infty):G_\infty(f,\beta)]$ must be finite whenever
  $[\Aut(T_\infty):G'_\infty(f,\beta)]$ is finite.  Furthermore, if
  $[K':K]$ is finite, then the following are equivalent:
\begin{enumerate}
\item $[\Aut(T_\infty):G_\infty(f,\beta)]<\infty$
\item $[\Aut(T_\infty):G'_\infty(f,\beta))]<\infty$
\item $[\Aut(T_\infty):G_\infty(g,\sigma(\beta))]<\infty$
\end{enumerate}
\end{prop}
\begin{proof}
  We have $|G'_n(f,\beta)| \leq |G_n(f,\beta)|$ by \eqref{any}, so,
  for any extension $K'$ of $K$, we see that
  %$[\Aut(T_\infty):G'_\infty(f,\beta))]<\infty$ must imply
  %$[\Aut(T_\infty):G_\infty(f,\beta)]<\infty$.  
  (2) implies (1). When $[K':K]$ is finite, we have
\[ |G'_n(f,\beta)| \geq \frac{|G_n(f, \beta)|}{[K':K]} \] by \eqref{finite}, 
and so (1) implies (2) as well.  Now, $[K':K(\beta)]$ is finite so by the same
reasoning, we see that
$[\Aut(T_\infty) : \Gal(K_\infty(g, \sigma(\beta)) \cdot K' / K'(\beta))] <
\infty$ if and only if
$[\Aut(T_\infty):G_\infty(g,\sigma(\beta))]<\infty$.  For each $n$,
we have $K_n(g,\sigma(\beta)) \cdot K' = K_n(f,\beta) \cdot K'$, since
$g^{-n}(\sigma(\beta)) = \sigma(f^{-n}(\beta))$ and $\sigma$ is
defined over $K'$.  Thus, (2) and (3) are equivalent.
\end{proof}

Applying Proposition~\ref{reductions} and changing coordinates, we may
assume in the proof of Theorem~\ref{thm: main} that our polynomial $f$
is of the form
\[
f(x)=x^3-3a^2x + b
\]
where $a,b\in K$, so that $\pm a$ are the critical points of $f$.
Furthermore, as noted before, we may assume in the function field case
that $K$ is a function field of transcendence degree one over
$\Qbar$.  

\begin{prop}\label{obstructions}
Let $f\in K(x)$ with $d=\deg f\geq 2$ and let $\beta\in K$. Suppose
that any of the following holds:
\begin{enumerate}
\item the map $f:\P^1(\Kbar)\to \P^1(\Kbar)$ is PCF;
\item the pair $(f,\beta)$ is not eventually stable;
\item $\beta$ is postcritical for $f$;
\item $f\in K[x]$, $d\geq 3$ and $f$ has precisely one finite critical
  point; or
\item $f\in K[x]$, $d=3$ and $f$ has distinct finite critical points $\gamma_1\neq\gamma_2$ 
such that there is some $n\geq 1$ with $f^n(\gamma_1)= f^n(\gamma_2)$.
\end{enumerate}
Then  $[\Aut(T_\infty):G_\infty(\beta)]=\infty$.
\end{prop}

\begin{proof}
For (1), see Theorem 3.1 in~\cite{RafeArborealSurvey}. For (2), see Proposition 2.2 and Proposition 3.3 in~\cite{RafeAlonEventualStability}.

For (3), if $f^n(\gamma)=\beta$ for some $n$ and critical point $\gamma$, then $T_\infty(\beta)$ has repeated vertices 
at the $n$th level and above. That is, $f^n(x)-\beta$ has at most $d^n-1$ roots in $\overline{K}$, whence $f^{n+1}(x)-\beta$ has at most $d^{n+1}-d$ roots, and in general, $f^{n+k}(x)-\beta$ has at most $d^{n+k}-d^k$ roots. So 
\[
|\Gal(K_{n+k}(\beta)/K_{n+k-1}(\beta))|\leq|S_d|^{d^{n+k-1}-d^{k-1}}.
\]
As $|\Aut(T_{n+k}/T_{n+k-1})|=(d!)^{d^{n+k-1}}$, it follows that 
\[
[\Aut(T_{n+k}/T_{n+k-1}):\Gal(K_{n+k}(\beta)/K_{n+k-1}(\beta))]\geq (d!)^{d^{k-1}}
\]
for $k\geq 0$, and so $[\Aut(T_\infty):G_\infty(\beta)]=\infty$.

To prove (4), observe that if $f$ has only one finite critical point
$\gamma$, then up to a change of variables (and replacing $K$ by a finite extension), 
$f(x)=x^d + c$ for some $c\in K$.  So $K_n(\beta)$ is an extension of 
$K_{n-1}(\beta)$ attained by adjoining an appropriate number of $d$th roots. 
Adjoining a $d$th root of unity to $K$ 
if necessary, it is easy to see that $\Gal(K_n(\beta)/K_{n-1}(\beta))$ is contained in the direct product of
$d^{n-1}$ copies of $C_d$, the cyclic group of order $d$, whereas $\Aut(T_n/T_{n-1})$ is the direct product of 
$d^{n-1}$ copies of $S_d$. As $d\geq 3$, $C_d$ is a proper subgroup of $S_d$. 
So $\Gal(K_n(\beta)/K_{n-1}(\beta))<\Aut(T_n/T_{n-1})$ for all $n$, and 
$G_\infty(\beta)$ cannot have finite index in $\Aut(T_\infty)$. 
See \cite{HamblenJonesMadhu} for more discussion
of iterated Galois groups of polynomials with one finite critical point.

In the situation where there is an $n$ such that $f^n(\gamma_1) = f^n(\gamma_2)$, 
where $\gamma_1$ and $\gamma_2$ are not preperiodic, Grell \cite{GG} has given an explicit description
of the profinite iterated monodromy group of $f$, following work of
Pink in the quadratic case \cite{PinkQuadraticInfiniteOrbits}. (If $\gamma_1,\gamma_2$ are preperiodic, then $f$ is PCF.) 
In particular, these groups have infinite index in $\Aut(T_\infty)$, and thus the iterated Galois
groups of such $f$ must as well.  This proves that
$[\Aut(T_\infty):G_\infty(\beta)]=\infty$ when (5) holds.  
\end{proof}

\section{Eventual Stability} \label{stable}

Recall the definitions of stability and eventual stability from the introduction. 
We establish some facts about eventually stable rational functions that will be used later. First 
we recall an basic result of algebra known as Capelli's Lemma. For a proof, see~\cite[p. 490]{Capelli}.

\begin{lem}[Capelli] 
Let $K$ be a field, $f(x)$, $g(x) \in K[x]$, and let $\alpha\in \overline{K}$ be any root of $g(x)$. Then $g(f(x))$ is irreducible over $K$ if and only if both $g$ is irreducible over $K$ and $f(x) - \alpha$ is irreducible over $K(\alpha)$.
\end{lem}

\begin{prop}\label{eventual stability facts}
Let $f\in K[x]$ and $\beta\in \Kbar$. Suppose that the pair $(f,\beta)$ is eventually stable. Then the following hold.
\begin{enumerate}
\item For all sufficiently large $n$ and any $\alpha\in f^{-n}(\beta)$, $f(x)-\alpha$ is irreducible over $K(\alpha)$. 
\item There exists $n$ such that for any $\alpha\in f^{-n}(\beta)$, the pair $(f,\alpha)$ is stable.
\item The point $\beta$ is not periodic for $f$.
\end{enumerate}
\end{prop}
\begin{proof}
Let $m$ be the maximum number of irreducible factors over $K(\beta)$ of $f^n(x)-\beta$ for any $n\geq 1$, so that
$$f^k(x)-\beta=g_1(x)\cdots g_m(x)$$
for some $k$ and some irreducible polynomials $g_1,\dots,g_m\in K(\beta)[x]$. Then for any $n \geq k$, we have
$$f^n(x)-\beta=g_1(f^{n-k}(x))\cdots g_m(f^{n-k}(x)),$$
and each $g_i(f^{n-k}(x))$ is irreducible over $K(\beta)$, as otherwise $m$ would 
not be the maximum number of irreducible factors of $f^n(x)-\beta$. 

Let $\theta_i=g_i\circ f^{n-k}$. For any $\alpha\in f^{-n}(\beta)$, we have $f^n(\alpha)=\beta$ 
and so $\theta_i(\alpha)=0$ for some $i$. Writing
$$f^{n+1}(x)-\beta = \theta_1(f(x))\cdots \theta_m(f(x))$$
shows that each $\theta_i(f(x))$ is irreducible over $K$. By Capelli's Lemma, $f(x)-\alpha$ is irreducible over $K(\alpha)$, proving (1).

Now let $n$ be large and let $\alpha\in f^{-n}(\beta)$. By (1), for any $m\geq 1$ and any $z\in f^{-m}(\alpha)$, we have $[K(z):K(f(z))]=\deg f$. Thus $[K(z):K(\alpha)]=(\deg f)^m$. Therefore $f^m(x)-\alpha$ is irreducible over $K(\alpha)$. So the pair $(f,\alpha)$ is stable, proving (2).

Finally, if $\beta$ is periodic for $f$, then there are infinitely many $n$ such that $f^n(\beta)=\beta$. If we set $\alpha=f(\beta)$, then $\alpha\in f^{-(n-1)}(\beta)$ for each of these $n$. But then $f(x)-\alpha$ is certainly not irreducible over $K(\alpha)$, because $\beta$ is a root, contradicting (1). This proves (3).
\end{proof}
\begin{rem}
In~\cite{RafeAlonEventualStability}, Jones and Levy conjecture that the converse of part (3) of 
Proposition~\ref{eventual stability facts} holds for all non-isotrivial $f$. 
If this conjecture holds, it would somewhat clarify the list of finite index 
obstructions in Theorem~\ref{thm: main}, as the eventual stability condition could be replaced by a 
simpler non-periodicity condition.
\end{rem}

\section{Height Estimates}\label{heights}
In this section we prove a variety of height estimates that will 
be used in the proof of Theorem~\ref{thm: main}. For background on heights, see 
~\cite{HindrySilverman,GNT,BridyTucker}. We set some notation below. 

If $K$ is a number field, let $\fo_K$ be its ring of 
integers. If $K$ is a function field, choose a prime $\fq$ of $K$ and set 
\[
\fo_K=\{z\in K:v_\fp(z)\geq 0\text{ for all }\fp\neq\fq\}.
\]
Let $\fp$ be a non-archimedean prime of $K$, and let $k_\fp$ be the residue field $\fo_K/\fp$. 
If $K$ is a number field, define
\[
N_\fp = \frac{1}{[K:\Q]}\log\#k_\fp,
\]
and if $K$ is a function field with field of constants $k$, instead define
\[
N_\fp =[k_\fp:k].
\]
For $z\in K$, let $h(z)$ denote the logarithmic height of $z$. For $f\in K[x]$ with $\deg f=d\geq 2$, 
let $h_f(z)$ denote the Call-Silverman canonical height of $z$ relative to $f$~\cite{CallSilverman}, defined by
\[
h_f(z) = \lim_{n\to\infty}\frac{h(f^n(z))}{d^n}.
\]
We will often write sums indexed by primes that 
satisfy some condition. These are taken to be primes of $\fo_K$. 
As an example of our indexing convention, observe that
\[
\sum_{v_\fp(z)>0} v_\fp(z)N_\fp\leq h(z)
\]
by the product formula for $K$.

We make use of the 
notion of \emph{good reduction} of a map $f\in K(x)$ at a prime $\fp$, which for a polynomial 
\[
f(x) = a_dx^d + a_{d-1}x^{d-1}+\dots + a_1x + a_0
\]
means that $v_\fp(a_d)=0$ and $v_\fp(a_i)\geq 0$ for $0\leq i\leq d-1$. 
See \cite{MortonSilverman} or \cite[Theorem 2.15]{SilvermanADS} for a 
more careful definition that also applies to rational functions. 
There are only finitely many primes $\fp$ such 
that $f$ has bad reduction (that is, does not have good reduction) at $\fp$. 
The main consequence of good reduction we will use is that if $f$ has good reduction at $\fp$, 
then $f$ commutes with the reduction mod $\fp$ map $\P^1(\Kbar)\to \P^1(\bar{k}_\fp)$. 
This is clear for polynomials; see \cite[Theorem 2.18]{SilvermanADS} for a proof in general. 
We further say that $f$ has \emph{good separable reduction} at $\fp$ if the reduced map  
$\bar{f}:\P^1(\bar{k}_\fp)\to\P^1(\bar{k}_\fp)$ is separable. 

\begin{lem}\label{B-delta}
Let $f\in K[x]$ with $d=\deg(f)\geq 2$. Let $\gamma, \beta \in K$ such that $h_f(\gamma)>0$ and 
$\beta\notin\mathcal{O}_f(\gamma)$.  
For any $\delta > 0$, there exists a constant $B_\delta$ with 
%\frac{\delta}{4d}\leq 
$\frac{1}{d^{B_\delta}}<\frac{\delta}{4}$
such that for any $n>B_\delta$, we have
\[ 
\sum_{m=1}^{n - B_\delta} \sum_{v_\fp(f^m(\gamma) - \beta) > 0} N_\fp \leq
\delta d^n h_f(\gamma).
\]
Moreover, for all sufficiently small $\delta$, we have $B_\delta\leq -\frac{\log\delta}{\log d}$.
\end{lem}
\begin{proof}

This is a variant of the inequality that underlies Proposition 5.1 of \cite{GNT}. By the product formula,
\begin{align*}
\sum_{m=1}^{n - B_\delta} \sum_{v_\fp(f^m(\gamma )-\beta) > 0} N_\fp\leq 
\sum_{m=1}^{n - B_\delta} h(f^m(\gamma)-\beta).
\end{align*}
Recall that the canonical height $h_f$ satisfies $h_f(f(z))=dh_f(z)$ and that
$\left|h(z)-h_f(z)\right |$ is uniformly bounded by a constant $C_f$ for all $z\in K$~\cite{CallSilverman}. 
Also note that $|h(z-\beta)-h(z)|$  is uniformly bounded by a constant 
$C_\beta$ depending only on $\beta$ for all $z\in K$. Choose $B_\delta$ such that 
$\frac{\delta}{4d}\leq\frac{1}{d^{B_\delta}}<\frac{\delta}{4}$ and 
$\frac{\delta}{2}d^nh_f(\gamma)> n(C_\beta+C_f)$ for all $n>B_\delta$. 
Note that it is always possible to choose such a $B_\delta$, as the conclusion of the Lemma 
holds upon replacing $\delta$ with any positive $\delta'\leq \delta$, and 
$\frac{n(C_\beta+C_f)}{d^nh_f(\gamma)}\to 0$ as $n\to\infty$.

If $n>B_\delta$, then
\begin{align*}
\sum_{m=1}^{n - B_\delta} h(f^m(\gamma)-\beta) & 
\leq \sum_{m=1}^{n - B_\delta} (h(f^m(\gamma))+C_\beta) \\
& \leq \sum_{m=1}^{n - B_\delta} (h_f(f^m(\gamma))+C_\beta+C_f) \\
& = n(C_\beta+C_f)+ \sum_{m=1}^{n - B_\delta} d^m h_f(\gamma) \\
& \leq n(C_\beta+C_f)+\frac{d^nh_f(\gamma)}{d^{B_\delta}}\sum_{r=0}^\infty \frac{1}{d^r}\\
& \leq n(C_\beta+C_f)+\frac{\delta}{2}d^n h_f(\gamma)\\
& \leq \delta d^nh_f(\gamma).
\end{align*}
We have $\frac{\delta}{4d}\leq\frac{1}{d^{B_\delta}}$, so $\log 4d -\log \delta\geq B_\delta\log d$, and 
\[
 B_\delta\leq -\frac{\log\delta}{\log d}.
\]
\end{proof}

Note that in the statements of Lemmas~\ref{from-5.1},~\ref{one-orbit}, and~\ref{odd}, the chosen 
$\beta_1$ and $\beta_2$ in $K$ need not be distinct.

\begin{lem}\label{from-5.1}
Let $f\in K[x]$ with $d=\deg(f)\geq 2$. Let $\gamma\in K$ with $h_f(\gamma)>0$.  Let
  $\beta_1, \beta_2 \in K$ such that $\beta_2\notin\O_f(\beta_1)$. 
  For $n>0$, let $\cX(n)$ denote the
  set of primes $\p$ of $\fo_K$ such that
\[ 
\min(v_\p(f^m(\gamma)-\beta_1), v_\fp(f^n(\gamma) - \beta_2)) > 0
\]
for some $0 < m < n$. Then for any $\delta >0$, we have
\begin{equation*}
\sum_{\p\in\cX(n)} N_\p\leq \delta d^n h_f(\gamma)+ O_\delta(1).  
\end{equation*}
for all $n$.
\end{lem}
\begin{proof}
For the finitely many primes $\fp$ such that $f$ has bad reduction at $\fp$, we can 
absorb any contribution to the sum of $N_\fp$ into $O_\delta(1)$, so we can 
assume that the primes in $\cX(n)$ are primes of good 
reduction for $f$. Then if
$$\min(v_\p(f^m(\gamma)-\beta_1), v_\fp(f^n(\gamma) - \beta_2)) > 0,$$
we have $f^m(\gamma)\equiv\beta_1\pmod{\fp}$ and $f^n(\gamma)\equiv\beta_2\pmod{\fp}$, so 
good reduction implies 
$f^{n-m}(\beta_1)\equiv\beta_2\pmod{\fp}$, equivalently $v_\fp(f^{n-m}(\beta_1)-\beta_2)>0$. Applying Lemma \ref{B-delta}, we have 
\begin{align*}
\sum_{\p\in\cX(n)} N_\p 
&\leq \sum_{m=1}^{n - B_\delta} \sum_{v_\fp(f^m(\gamma) - \beta_1) > 0} N_\fp + 
\sum_{m=n-B_\delta+1}^{n-1} \sum_{v_\fp(f^{n-m}(\beta_1)-\beta_2)>0} N_\fp + O_\delta(1) \\
& \leq \delta d^nh_f(\gamma) + 
\sum_{i=1}^{B_\delta-1}\sum_{v_\fp(f^i(\beta_1)-\beta_2)>0}N_\fp + O_\delta(1).
\end{align*}
Because $f^i(\beta_1)-\beta_2\neq 0$ for all $i$, the remaining sum of $N_\fp$ comprises a finite set of primes $\fp$ depending on $B_\delta$ (and thus on $\delta$), so it can be absorbed into the $O_\delta(1)$ term. 
\end{proof}

\begin{lem}\label{one-orbit}
Let $f\in K[x]$ with $d=\deg(f)\geq 2$. 
  Let $\gamma_1, \gamma_2 \in K$ with $h_f(\gamma_2)>0$, and suppose
  there are integers $\ell_1 > \ell_2$ such that
  $f^{\ell_1}(\gamma_1) = f^{\ell_2}(\gamma_2)$.  Let
  $\beta_1, \beta_2 \in K$ with $\beta_2\notin\O_f(\beta_1)$.  For $n>0$, let $\cY(n)$ denote
  the set of primes $\p$ such that
\[ 
\min(v_\p(f^m(\gamma_1)-\beta_1), v_\fp(f^n(\gamma_2) - \beta_2)) > 0
\]
for some $0 < m \leq n$. Then for any $\delta >0$, we have
\begin{equation*}
\sum_{\p\in\cY(n)} N_\p\leq \delta d^n h_f(\gamma_2)+ O_\delta(1).  
\end{equation*}
for all $n$.
\end{lem}
\begin{proof}
The finitely many primes that contribute to the sum for $n<\ell_1$ or $m<\ell_1$ can be 
absorbed into the $O_\delta(1)$ term, as can the primes $\fp$ for which 
$f$ has bad reduction at $\fp$. So assume that $f$ has good reduction at $\fp$, $n\geq \ell_1$, $m\geq\ell_1$, $f^m(\gamma_1)\equiv \beta_1\pmod{\fp}$, and $f^m(\gamma_1)\equiv\beta_2\pmod{\fp}$ for some $m\leq n$. Let $m' = m-\ell_1+\ell_2$. Then
\[
f^{m'}(\gamma_2)= f^{m-\ell_1}(f^{\ell_2}(\gamma_2)) = f^{m-\ell_1}(f^{\ell_1}(\gamma_1)) =f^m(\gamma_1)\equiv\beta_1\pmod{\fp}.
\]
So $\min(v_\fp(f^{m'}(\gamma_2)-\beta_1),v_\fp(f^{n}(\gamma_2)-\beta_2))>0$, and also   
$0<m'<n$ because $\ell_2-\ell_1<0$. We are
done by applying Lemma \ref{from-5.1} to bound the contribution to the sum from such $\fp$.
\end{proof}

Lemma~\ref{odd} will be used to treat the case of odd polynomials, 
which require special attention.

\begin{lem}\label{odd}
Let $f\in K[x]$ with $d=\deg(f)\geq 3$.   Suppose that $f$ is odd
(i.e. that $f(-x) = - f(x))$.  Let $\gamma\in K \setminus \{0\}$ with $h_f(\gamma)>0$.  Let
  $\beta_1, \beta_2 \in K \setminus \{0\}$ be such that $-\beta_1\neq\beta_2$ 
  and $\beta_2\notin\O_f(-\beta_1)$. For $n>0$, let $\cZ(n)$ denote the
  set of primes $\p$ such that
\[ 
\min(v_\p(f^m(-\gamma)-\beta_1), v_\fp(f^n(\gamma) - \beta_2)) > 0
\]
for some $0 < m \leq n$. Then for any $\delta >0$, we have
\begin{equation*}
\sum_{\p\in\cZ(n)} N_\p\leq \delta d^n h_f(\gamma)+ O_\delta(1).  
\end{equation*}
for all $n$.
\end{lem}
\begin{proof}
  Since $f$ is odd, we have $f^m(-\gamma) = - f^m(\gamma)$ for all
  $m$.    Thus, if
\[ \min(v_\p(f^n(-\gamma)-\beta_1), v_\fp(f^n(\gamma) - \beta_2)) >
  0\]
then $v_\fp(\beta_1 + \beta_2) > 0$, which restricts such $\fp$ to a
finite set as $\beta_1+\beta_2\neq 0$.   Since $v_\fp(f^m(-\gamma) -\beta_1) = v_\fp(f^m(\gamma) -
(-\beta_1))$ for any prime $\p$, applying Lemma \ref{from-5.1}  with $-\beta_1$ in place of $\beta_1$
finishes the proof.  
\end{proof}

Lemma~\ref{from-Roth} is the ``Roth-$abc$" estimate mentioned in the introduction. For a number field $K$, 
it is conditional on the $abc$ conjecture for $K$. As we do not make further use of the $abc$ conjecture, 
we refer the reader to~\cite{BridyTucker} for its precise statement.

\begin{lem}\label{from-Roth}
  Let $f\in K[x]$ with $d=\deg(f)\geq 3$. Let $\gamma,\beta\in K$ be such that
  $\beta\notin\O_f(\gamma)$. If $K$ is a number field, assume the $abc$-conjecture for $K$. 
  If $K$ is a function field, assume that $f$ is not isotrivial. 
  Then for every $\epsilon > 0$, there exists a constant
  $C_\epsilon$ such that
\[
\sum_{v_\p(f^n(\gamma)-\beta)= 1}N_\p\geq
(d-\epsilon)d^{n-1}h_f(\gamma) + C_\epsilon.
\]
\end{lem}
\begin{proof}
See Propositions 3.4 and 4.2 in~\cite{GNT}.
\end{proof}

\section{Ramification and Galois theory}\label{Galois section}
Let $f(x) = x^3 - 3a^2 x + b$ with $a,b\in K$. 
%Without loss of generality, suppose that $h_f(a) \geq h_f(-a)$. 
In this section we define \emph{Condition R} 
and \emph{Condition U} in terms of primes dividing certain elements of $K$ 
related to the forward orbits of $a$ and $-a$. In Proposition~\ref{necessary} and~\ref{main Galois} we 
show that these conditions control ramification in the extensions
$K(\beta) \subseteq K_n(\beta)$, with 
consequences for the Galois theory of these extensions. 
We begin with the following 
standard lemma from Galois theory.

\begin{lem}\label{Galois}
Let $L_1,\dots,L_n$ and $M$ be fields all contained in some larger field. Assume that $L_1,\dots,L_n$ are finite extensions of $M$.
\begin{enumerate}
\item [(i)]  If $L_1,L_2$ are Galois over $M$ with $L_1\cap L_2=M$, then $L_1 L_2$ is Galois over $L_2$ and
$\Gal(L_1 L_2/L_2)\cong\Gal(L_1/M)$.
\item [(ii)] If $L_1,\dots,L_n$ are Galois over $M$ with $L_i\cap\prod_{j\neq i}L_j = M$ for each $i$, then $\Gal(\Pi_{i=1}^n L_i/M)\cong\prod_{i=1}^n\Gal(L_i/M)$.
\end{enumerate}
\end{lem}

\begin{defn}
  Let $\beta \in \Kbar$.  We say that a
  prime $\fp$ of $K(\beta)$ satisfies {\bf Condition R} at
  $\beta$ for $n$ if the following hold:
\begin{enumerate}
\item[(a)] $f$ has good separable reduction at $\fp$;
\item[(b)] $v_\fp(f^i(-a) - \beta) = 0$ for all $0 \leq i \leq n$;
\item[(c)]  $v_\fp(f^i(a) - \beta) = 0$ for all $0 \leq i < n$;
\item[(d)]  $v_\fp(f^n(a) - \beta) = 1$;
\item[(e)] $v_\fp(\beta) = 0$.
\end{enumerate}
\end{defn}

\begin{defn}
  Let $\beta \in \Kbar$.  We say
  that a prime $\fp$ of $K(\beta)$ satisfies {\bf Condition U} at
  $\beta$ for $n$ if the following hold:
\begin{itemize}
\item[(a)] $f$ has good separable reduction at $\fp$;
\item[(b)] $v_\fp(f^i(a) - \beta) =  v_\fp(f^i(-a) - \beta) = 0$ for all $0 \leq i \leq n$;
\item[(c)] $v_\fp(\beta)=0$.
\end{itemize}
\end{defn}

\begin{prop}\label{necessary}
Let $\beta\in\Kbar$. Let $\p$ be a prime of $K(\beta)$ that satisfies condition U at $\beta$ for $n$. 
Then $\fp$ is unramified in $K_n(\beta)$.
\end{prop}
\begin{proof}
This follows immediately from \cite[Proposition 3.1]{BridyTucker}. The proof in \cite{BridyTucker} is 
stated for $\beta\in K$, but nothing changes if we allow $\beta\in\Kbar$ and replace $K$ with $K(\beta)$.
\end{proof}

\begin{prop}\label{main Galois}
Let $\beta\in\Kbar$. Suppose that $\fp$ is a prime of $K(\beta)$ that 
satisfies Condition R at $\beta$ for $n$ and that $f^n(x)
- \beta$ is irreducible over $K(\beta)$. Then 
\[ \Gal( K_n(\beta) / K_{n-1}(\beta)) \cong (S_3)^{3^{n-1}},\]
the direct product of $3^{n-1}$ copies of $S_3$. Furthermore, $\fp$
does not ramify in $K_{n-1}(\beta)$ and any field $E$ such that
$K_{n-1}(\beta) \subsetneq E \subset K_n(\beta)$ must
ramify over $\fp$.
\end{prop}
\begin{proof}
  Since Condition R at $\beta$ for $n$ implies Condition U at $\beta$
  for $n-1$, Proposition \ref{necessary} implies that $\fp$ does not
  ramify in $K_{n-1}(\beta)$.  Now, consider the map ${\bar f}$ on
  $\P^1(k_\fp)$ that comes from reducing $f$ at $\fp$.  For any
  $z \in K$ with $v_\fp(z) \geq 0$, we let ${\bar z}$ denote its
  reduction at $\fp$.  The critical points of ${\bar f}$ are then
  ${\bar a}$ and $-{\bar a}$.  Then, by (c) and (d) of Condition R, we
  see that every point in ${\bar f}^{-(n-1)} ({\bar \beta})$ has
  ramification index one over ${\bar \beta}$ with respect to
  ${\bar f}^{n-1}$.  Since
  ${\bar f}({\bar a}) \in {\bar f}^{-(n-1)} ({\bar \beta})$ and
  ${\bar f}(- {\bar a}) \notin {\bar f}^{-(n-1)} ({\bar \beta})$ by
  (b) of Condition R, we see then that ${\bar a}$ has multiplicity 2
  as a root of ${\bar f}^n(x) - {\bar \beta}$ and that
  ${\bar f}^n(x) - {\bar \beta}$ has no multiple roots besides
  ${\bar a}$.  Thus, we see the reduction of $f^n(x) - \beta$ at $\fp$
  can be written as
\begin{equation} \label{h}
\bar{f}^n(x)-\bar{\beta} = (x - {\bar a})^2 h(x)
\end{equation}
where $h \in k_\fp[x]$ is coprime to $(x-{\bar a})$ and has no repeated
roots.  

Let $z_1, \dots, z_{3^{n-1}}$ be the roots of $f^{n-1}(x) - \beta$.
Then $f^n(x) - \beta$ factors as
\begin{equation}\label{factor}
f^n(x) - \beta=\prod_{i=1}^{3^{n-1}} f(x) - z_i 
\end{equation}
in $K_{n-1}(\beta)$.  For each $j$, let $L_j$ denote the splitting field of $f(x)-z_j$ over 
$K(z_j)$, and let $M_j$ denote the splitting
field of $f(x) - z_j$ over $K_{n-1}(\beta)$. We see that $M_j = K_{n-1}\cdot L_j$. 
By Capelli's Lemma, each $f(x)-z_j$ is irreducible over $K(z_j)$.

Note that the $z_i$ are distinct modulo $\fp$ since ${\bar f}^{n-1}$
does not ramify at any ${\bar z_i}$, as noted above.  Let $\fq$ be a
prime of $K_{n-1}(\beta)$ that lies over $\fp$, and let $\fm$ be a
prime of $M_i$ lying over $\fq$. By \eqref{factor} and \eqref{h},
there is exactly one $i$ such that $f(x) - z_i$ has multiple roots
modulo $\fq$, and at this $i$, the polynomial $f(x) - z_i$ has a root
of multiplicity 2 modulo $\fq$.  Then, for $k \not=i$, we see that
$\fq$ does not ramify in $M_k$.  As for its ramification in $M_i$, we
have $f(x) - z_i \equiv (x - w_1)^2 (x-w_2)\pmod{\fm}$ where
$w_1 \equiv a \pmod{\fm}$ and $w_2 \not\equiv a \pmod{\fm}$.  Now,
$v_\fq(f(a) - z_i) = 1$ since $v_\fp(f^n(a) - \beta) = 1$ by (d) of
Condition R and $\fq$ does not ramify over
$\fp$.  On the other hand, $v_\fm(f(a) - z_i) = 2
v_\fm(a-w_1)$, so we see that $\fm$ ramifies over
$\fp$ with ramification index 2.  As $\fm$ is unramified in
$K(z_i)$, we have that $\fm\cap K(z_i)$ also ramifies over
$\fp$ with ramification index 2. Since $f(x) -
z_i$ is irreducible over $K(z_i)$, this implies that
$[L_i:K(z_i)]=6$, i.e.  $\Gal(L_i/K(z_i))\cong
S_3$. Then $\Gal(M_i/K_{n-1}(\beta))$ is a normal subgroup of
$S_3$ (recall $M_i=L_i\cdot
K_{n-1}(\beta)$) that contains the order 2 inertia subgroup
$I(\fm/\sigma \fq)$, and so $\Gal(M_i/K_{n-1}(\beta))\cong
S_3$ as well.  Thus,
$\Gal(M_i/K_{n-1}(\beta))$ is generated by the three conjugates of
$I(\fm/\sigma \fq)$, and any field $E'$ such that $K_{n-1}(\beta)
\subsetneq E' \subset M_i$ must ramify over $\fq$.

Since $f^{n}(x) - \beta$ is irreducible over $K(\beta)$, it follows easily that 
$f^{n-1}(x)-\beta$ is irreducible over $K(\beta)$. So all of the $z_k$ for which
$f^{n-1}(z_k) = \beta$ are conjugate to each other, i.e. for each
$1\leq i,j\leq 3^{n-1}$, there exists $\sigma\in G_{n-1}(\beta)$ such that $\sigma z_i = z_j$.  Then
$\sigma \fq$ ramifies in $M_j$ with ramification index 2 and does 
not ramify in $M_k$ for $k \not = j$.  As above,
$\Gal(M_j/K_{n-1}(\beta))$ is isomorphic to $S_3$ and is 
generated by inertia groups of the form $I(\fm/\sigma \fq)$ for
primes $\fm$ of $M_j$ lying over $\sigma \fq$.  Hence any field $E'$
such that $K_{n-1}(\beta) \subsetneq E' \subset M_j$ must ramify
over $\sigma \fq$.  Since $\sigma \fq$ does not ramify in $M_k$ for
$k \not= j$, we must therefore have
$M_j \cap \prod_{k \not= j} M_k =K_{n-1}(\beta)$.  Thus by
Lemma \ref{Galois} the Galois group
$\Gal(K_{n}(\beta)/K_{n-1}(\beta))$ is isomorphic to
$(S_3)^{3^{n-1}}$.  

Each inertia group $I(\fm/\sigma \fq)$ for $\fm$ a
prime of $M_j$ extends to an inertia group $I(\fm' /\sigma \fq)$ for a
prime $\fm'$ of $K_n(\beta)$ lying over $\fm$; the inertia group
$I(\fm' /\sigma \fq)$ restricts to $I(\fm/\sigma \fq)$ on $M_j$ and to
the identity on $M_k$ for $k \not= j$.  Thus,
$\Gal(K_{n}(\beta)/K_{n-1}(\beta))$ is generated by
inertia groups of the form $I(\fm/\fq')$ for primes $\fm'$ of
$K_n(\beta)$ and primes $\fq'$ of $K_{n-1}(\beta)$ lying over
$\fp$.  Hence, any field $E$ such that
$K_{n-1}(\beta) \subsetneq E \subset K_n(\beta)$ must
ramify over a prime of $K_{n-1}(\beta)$ lying over $\fp$ and
thus must ramify over $\fp$, as desired.
\end{proof}

Before stating the last result from Galois theory we need, we need a
little notation. 

\begin{defn}
  Let $K$ be a field, let $f \in K[x]$, and let
  $\ba = (\alpha_1, \dots, \alpha_s)$, where $\alpha_i \in \Kbar$.  We
  define $K_n(\ba)$ to be $ \prod_{i=1}^s K(f^{-n}(\alpha_i))$.
\end{defn}

With this notation we have the following.

\begin{prop}\label{final Galois}
  Let $\ba = (\alpha_1, \dots, \alpha_s) \subseteq L$ for $L$ a
  finite extension of $K$.  Suppose there exist primes
  $\fp_1,\dots,\fp_s$ of $L$ such that
\begin{itemize}
\item[(a)] $\fp_i \cap K(\alpha_i)$ satisfies Condition R at
$\alpha_i$ for $n$;
\item[(b)] $\fp_i \cap K(\alpha_j)$ satisfies Condition U at
$\alpha_j$ for $n$ for all $j \not= i$;
\item[(c)] $\fp_i \cap K(\alpha_i)$ does not ramify in $L$; and
\item[(d)] $f^n(x) - \alpha_i$ is irreducible over $K(\alpha_i)$ for
  $i = 1, \dots s$.
\end{itemize}
Then  $\Gal(K_n(\ba)/ K_{n-1}(\ba)) \cong S_3^{s3^{n-1}}$.
\end{prop}
\begin{proof}
By Proposition \ref{main Galois}, for $i= 1, \dots, s$, we have 
\[ \Gal(K_n(\alpha_i) / K_{n-1} (\alpha_i) )\cong S_3^{3^{n-1}} .\]
Also by Proposition \ref{main Galois}, for each $i$, the
prime $\fp_i$ does not ramify in $K_{n-1}(\alpha_i)$, but any field $E$ with
$K_{n-1}(\alpha_i)\subsetneq E \subset K_n(\alpha_i)$ must ramify over
$\fp_i$.   Let $L_i$ denote $K_n(\alpha_i) \cdot K_{n-1}(\ba)$.  By
Proposition \ref{necessary} and condition (c), the prime
$\fp_i$ does not ramify in $K_{n-1}(\ba)$ or in $L_j$ for $j \not= i$,
so we must have
$K_n(\alpha_i) \cap \prod_{j \not= i} L_j = K_{n-1}(\alpha_i)$ and
$K_n(\alpha_i) \cap K_{n-1}(\ba) = K_{n-1}(\alpha)$.  Thus,
\[ \Gal(L_i \cdot \prod_{j \not= i} L_j /\prod_{j \not= i} L_j)  \cong
  \Gal(L_i / K_{n-1}(\ba))
  \cong \Gal(K_n(\alpha_i) / K_{n-1} (\alpha_i) ) \cong S_3^{3^{n-1}} \]
by Lemma \ref{Galois}(i), since $L_i \cdot \prod_{j \not= i} L_j=
K_n(\alpha_i) \cdot \prod_{j \not= i} L_j$.
This means that 
\[ 6^{3^n-1} =  [L_i :
  K_{n-1}(\ba)] \geq [L_i : L_i \cap \prod_{j \not= i} L_j]  \geq  [L_i \cdot \prod_{j \not= i} L_j: \prod_{j
    \not= i} L_j] = 6^{3^n -1} .\]
  
Thus, we have $L_i \cap \prod_{j \not= i} L_j = K_{n-1}(\ba)$.  Using
the fact that $L_i \cdot \prod_{j \not= i} L_j = K_n(\ba)$ and
applying Lemma \ref{Galois}(ii) then finishes our proof.
 
\end{proof}

\section{The case of number fields}\label{number-case}
Throughout this section, $K$ will denote a number field. The main
result of the section is Lemma~\ref{both critical points}, which
handles the height estimates needed for the proof of Theorem~\ref{thm:
  main}. We will use the following form of Vojta's conjecture for
$K$~\cite[Conjecture 3.4.3]{VojtaBook}.

\begin{conjecture}[Vojta]\label{Vojta}
  Let $V$ be a smooth projective variety over a number field. Let $\cK$ be the
  canonical divisor of $V$ and let $A$ be an ample normal crossings
  divisor, with associated height functions $h_\cK$ and $h_A$.  For any
  $\epsilon>0$, there is a proper Zariski-closed $X_\epsilon\subset V$
  and a constant $C_\epsilon=C_\epsilon(V,K,A)$ such that
\[
h_{\cK}(x)\leq h_A(x)+C_\epsilon
\]
for all $x\in V(K)\setminus X_\epsilon(K)$.
\end{conjecture}

A sequence $(a_n)$ of points on a variety $V$ is said to be \emph{generic} if for every proper Zariski-closed subset $X\subset V$, there exists $N\in\N$ such that $a_n\notin X$ for all $n>N$. Note that this 
is a stronger condition than the set of all points in the sequence being Zariski-dense in $V$. 
A key step in our proof for number fields uses Theorem~\ref{dynamical gcd bound}, 
a conditional bound on ``dynamical gcds" for generic orbits due to Huang~\cite[Theorem 2.4]{Keping}. 
For any valuation $v\in M_K$, let $v^+(a)=\max(v(a),0)$. For $a,b\in K$, define
\[
h^0_{gcd}(a,b) = \sum_{\fp\text{ of }\fo_K} N_\fp \min(v^+_\fp(a),v^+_\fp(b))
\]
and
\[
h_{gcd}(a,b) = h^0_{gcd}(a,b) + \frac{1}{[K:\Q]}\sum_{v\in M_K^\infty}\min(v^+(a),v^+(b)),
\]
where $N_\fp$ is as in Section~\ref{heights} and $M_K^\infty$ is the set of archimedean places of $K$.
Observe that $h^0_{gcd}(a,b)$ is a finite sum where 
the only positive contributions come from primes $\fp$ such that both $v_{\fp}(a)>0$ and $v_{\fp}(b)>0$. 
Clearly $h^0_{gcd}(a,b)\leq h_{gcd}(a,b)$. 

The point $z\in \P^1(\Kbar)$ is said to be 
\emph{exceptional} for $f\in K(x)$ if the backward orbit of $z$ under $f$ is finite, i.e. if there are only 
finitely many $y\in\P^1(\Kbar)$ such that $f^n(y)=z$ for some $n\geq 1$. It is easy to show that, up to conjugation 
by M\"{o}bius transformations, $z$ is exceptional for $f$ if and only if either $f$ is a polynomial and $z=\infty$ or 
$f(x)=x^d$ for some nonzero $d\in\Z$ and $z\in\{0,\infty\}$. With this notation, we can state 
Huang's result, which may be seen as a dynamical analog of work of
Bugeaud, Corvaja, and Zannier \cite{BCZ, CZ05}.  We note that Huang's
proof follows ideas of Silverman \cite{SilBlowup}, which connect the
result of Bugeaud, Corvaja, and Zannier with a special case of Vojta's
conjecture.  

\begin{thm}\label{dynamical gcd bound}
Assume Vojta's conjecture for blowups of $\P^1\times\P^1$. 
Let $f,g\in K(x)$ with $\deg f=\deg g$ and let $a,b,c,d\in K$ such that 
$c$ is not exceptional for $f$ and $d$ is not exceptional for $g$. 
Suppose that the sequence $(f^n(a),g^n(b))$ is generic in $\A^2(K)$. 
Then for every $\epsilon>0$, there exists a constant $C_\epsilon=C_\epsilon(a,b,c,d,f,g)$ such that
$$h_{gcd}(f^n(a)-c,g^n(b)-d)<\epsilon\max(h_f(a),h_g(b))(\deg f)^n+C_\epsilon$$
for all $n\geq 1$.
\end{thm}

If the sequence $(f^n(a),f^n(-a))$ is generic in $\A^2(K)$, then 
we will apply Theorem~\ref{dynamical gcd bound} to 
demonstrate the existence of primes that satisfy Condition R of Section~\ref{Galois}. 
We also need to handle the possibility that this sequence is not generic. This is done using 
the following result of Xie~\cite{XieDMLA2}, which proves the Dynamical Mordell-Lang 
Conjecture (see e.g.~\cite{GTZ,CaseOfDML}) for polynomial endomorphisms of $\A^2(\overline{\Q})$. 
This will be used to prove Lemma~\ref{not generic}, which shows that in the non-generic case, 
the sequence must lie on a curve of a very restricted form.

\begin{thm}\label{DML}
  Let $F:\A^2 \to\A^2$ be a polynomial endomorphism defined over
  $\overline{\Q}$. Let $C$ be an irreducible curve in $\A^2$ and let
  $\alpha$ be a closed point in $\A^2(\overline{\Q})$. Then the set
  $\{n\in\N : F^n(\alpha)\in C\}$ is a finite union of arithmetic
  progressions.
\end{thm}

Let $f\in K[x]$ be a polynomial of degree $d\geq 2$. We say $f$ is in \emph{normal form} if $f$ is monic and the second highest degree term of $f$ has a coefficient of zero. That is, a normal form $f$ can be written
$$f(x)=x^d+a_{d-2}x^{d-2}+a_{d-3}x^{d-3}+\dots+a_1x+a_0.$$
It is easy to show for any nonlinear $f\in K[x]$, there exists a linear $\sigma\in \overline{K}[x]$ such that $\sigma\circ f\circ \sigma^{-1}(x)$ is in normal form.

We will use Lemma~\ref{normal form} to determine the linear polynomials that commute with iterates of normal form polynomials. 
%For this lemma, $K$ can be any field of characteristic $0$. 
We say $f$ is of \emph{gap} $k$ if the second highest degree term in $f$ is of degree $\deg f-k$.
\begin{lem}\label{normal form}
Suppose that $f\in K[x]$ of degree $d\geq 2$ is in normal form of gap
$k$. Then $f^n$ is in normal form of gap $k$ for all $n\geq 1$.
\end{lem}
\begin{proof}
  Write 
\[ f(x)=x^d+ a_{d-k}x^{d-k} + \text{terms of degree
  less than $d - k$}. \]
By induction we have
\[ f^n(x) = x^{d^n} + d^{n-1}a_{d-k}x^{d^n-k} + \text{terms of degree
  less than $d^n - k$}. \]
\end{proof}

\begin{lem}\label{not generic}
  Let $f(x)=x^3-3a^2x + b\in K[x]$ with $a \not= 0$ and $b \not= 0$.  Suppose that $f$ is not PCF. 
  If the sequence $(f^n(a),f^n(-a))$ is not
  generic in $\A^2(K)$, then there exist $m,n\in\N$ such that
  $f^m(a) = f^n(-a)$.
\end{lem}
\begin{proof}
Let $F: \A^2 \lra \A^2$ be given by $F(x,y) = (f(x), f(y))$.  
If the sequence $F^n(a, -a)$ is not
  generic in $\A^2$, there is a curve $C\subseteq \A^2$ such that
  $F^n(a, -a)\in C$ for infinitely many $n$. By Xie's 
  proof \cite{XieDMLA2} of the dynamical Mordell-Lang
  conjecture for $\A^2(K)$, it follows that there are integers
  $m_0\geq 0$ and $m_1>0$ such that $F^{m_0 + k m_1}(a,-a) \in C$ for
  all integers $k \geq 0$. Therefore
  $\{F^{m_0+m_1k}(a, -a): k\geq 0\}\subseteq C$, and the Zariski closure
  of this set is $F^{m_1}$-invariant. As $f$ is not PCF, this set is infinite, so $C$ is
  $F^{m_1}$-invariant.

  By the chain rule, critical points of $f$ are also critical points
  of $f^{m_1}$. Thus, as we assumed $f$ is not PCF, it follows that
  $f^{m_1}$ is not PCF either. In particular, $f^{m_1}$ is not a power
  map, Chebyshev polynomial, or negative Chebyshev polynomial. Therefore by
  work of Medvedev and Scanlon~\cite{MedvedevScanlon}, the $F^{m_1}$-invariant curve $C$ is the
  graph of a polynomial that commutes with an iterate of $f$. That is, $C$ is
  given by an equation of the form $y=p(x)$ or $x=p(y)$ where
  $p\in K[x]$ commutes with with some iterate of $f$, that is, 
  $f^t\circ p=p\circ f^t$ for some $t>0$. Moreover, Ritt's
  theorem on commuting polynomials \cite{Ritt1} implies that
  $p(x)=L\circ h^r(x)$ for some $r>0$ and $L,h\in K[x]$ where
  $h^k=f^\ell$ for some $k$ and $\ell$, and $L$ is a linear polynomial
  that commutes with some iterate of $f$.  By Lemma~\ref{normal form}, as $f$ is 
  in normal form with gap 2, $f^n$ is also in normal
  form with gap 2. So the only nonidentity linear polynomial $L$ that can
  commute with any iterate of $f^n$ is $L(x) = -x$.  Now, by \cite{Reznick}, if $f^n$
  is odd, then $f$ is odd, but we have assumed that $b \not= 0$, so $f$ is not odd.
  Thus, any polynomial that commutes with $f$ shares a common iterate
  with $f$.

Ritt's classification of complex polynomials with a common iterate \cite{Ritt2} gives
\[f(x)=-B+\epsilon_1g^{n_1}(x+B)\]
and
\[h(x)=-B+\epsilon_2g^{n_2}(x+B)\]
where $B\in\C$, $g(x)\in x^r\C[x^s]$ for some $s\geq 1$, and
$\epsilon_1^s=\epsilon_2^s=1$. Note that $s$ can be taken to be
the largest integer such that the exponents occurring in $g$ form an
arithmetic progression of modulus $s$.  We must have $n_1 = 1$, since
$f$ has prime degree.  Furthermore, we must have $s = 1$, since if
$s \geq 2$, then $B = 0$ (since $f$ is in normal form
already), which would mean that $g$ is a multiple of $f$, and the linear
term of $f$ would be zero, contradicting our assumption that $a\neq 0$.  It follows then that
$\epsilon_1 = \epsilon_2 = 1$.  Now, let $\sigma(x) = x + B$.  Then
$f = \sigma^{-1} g \sigma$ and $h = \sigma^{-1} g^{n_2}\sigma$, so
$h = f^{n_2}$.  Therefore, $C$ is given by an equation of the form
$y= f^r(x)$ or $x= f^r(y)$.  Since $(f^{m_0}(a), f^{m_0}(-a)) \in C$,
we see then that there are $m$, $n$ such that $f^m(a) = f^n(-a)$, as
desired. 
\end{proof}

%%% May need to fix
%\begin{prop}\label{number}
%Let $c,d \in K$, and suppose that $h_f(a) > 0$.  Let $\cY(n)$ be
%the set of finite primes $v$ of $K$ such that $\min(v(f^n(a) - c),
%v(f^m(-a) - d)) >0$ for some $0 < m < n$.  Then, for any $\delta > 0$, we have
%\[ \sum_{v \in \cY(n)} N_v \leq \delta h(f^n(a)) + O_\delta(1) \]
%for all $n$.
%\end{prop}

\begin{lem}\label{both critical points}
Let $f(x)=x^3-3a^2x+b\in K[x]$ with $a\neq 0$ and $b\neq 0$. Let $\beta_1,\beta_2\in K$. 
Assume both the $abc$-conjecture for $K$ and Vojta's conjecture for blowups 
of $\P^1 \times\P^1$. Suppose that $f$ is not PCF and that $h_f(a)\geq h_f(-a)$. 
Suppose that $f^i(a)\neq\beta_1$ 
and $f^i(-a)\neq\beta_2$ for any $i>0$. Let $\cY(n)$ be the 
set of primes $\fp$ of $K$ such that 
\[
\min(v_\fp(f^n(a)-\beta_1),v_\fp(f^m(-a)-\beta_2))>0
\]
for some $1\leq m\leq n$. Then for any $\delta>0$, we have
\[
\sum_{\fp\in \cY(n)}N_\fp\leq \delta 3^n h_f(a) + O_\delta(1).
\]
\end{lem}
\begin{proof}
First note that $h_f(a)>0$, because if $h_f(a)=0$ then $h_f(-a)=0$ also. But Northcott's Theorem 
then implies $a$ and $-a$ are both preperiodic for $f$ (see ~\cite{CallSilverman}). Then we have that $f$ is PCF, 
which is excluded by assumption.

If $\beta_1$ or $\beta_2$ were exceptional for $f$, then $f$ would be a power map, again 
contradicting our assumption that $f$ is not PCF. Thus, if  
the sequence $(f^n(a),f^n(-a))$ is not generic in $\A^2$, then by 
Theorem \ref{not generic} there are positive integers $\ell_1,\ell_2$ such that 
$f^{\ell_1}(a)=f^{\ell_2}(-a)$, and the conclusion follows from Lemma \ref{one-orbit}. 
Assume therefore that the sequence $(f^n(a),f^n(-a))$ is generic in $\A^2$. 
Theorem \ref{dynamical gcd bound} and the inequality $h^0_{gcd}\leq h_{gcd}$ 
imply that the conclusion holds in the case where $m=n$, that is,
\[
\sum_{\min(v_\fp(f^n(a)-\beta_1),v_\fp(f^n(-a)-\beta_2))>0}N_\fp\leq \delta 3^n h_f(a) + O_\delta(1).
\]
For $m<n$, a more delicate argument is required. 

By Lemma \ref{B-delta}, up to possibly replacing $\delta$ by something smaller, 
there exists $B_\delta\leq-\frac{\log\delta}{\log 3}$ such that
\begin{equation}\label{one}
\sum_{m=1}^{n - B_\delta}
 \sum_{v_\fp(f^m(-a) - \beta_2) > 0} N_\fp \leq \delta 3^n h_f(-a)\leq \delta 3^nh_f(a). 
\end{equation}
Let $m$ be such that $n-B_\delta<m\leq n$. Assume that $f$ has good reduction at $\fp$ (as usual, any contribution from the finitely many primes of bad reduction will be absorbed into the $O_\delta(1)$ term). Then if $v_\fp(f^m(-a) - \beta_2) > 0$, we have $f^m(-a)\equiv\beta_2\pmod{\fp}$, so 
\[
f^n(-a)=f^{n-m}(f^m(-a))\equiv f^{n-m}(\beta_2)\pmod{\fp}
\]
where $n-m<B_\delta$. For each $j$ with $0\leq j<B_\delta$, let $\cW(j)$ be the set 
of primes $\fp$ of $\fo_K$ such that $\min(v_\fp(f^n(a)-\beta_1),v_\fp(f^n(-a)-f^j(\beta_2))>0$. Then
\[
\sum_{\min(v_\fp(f^n(a)-\beta_1),v_\fp(f^m(-a)-\beta_2))>0} N_\fp\leq
\sum_{\p\in \cW(n-m)}N_\fp.
\]
Now we apply Theorem 
\ref{dynamical gcd bound} to bound the contribution to the sum from the various $\cW(j)$:
\begin{align}\label{two}
 \sum_{j=1}^{B_\delta-1}\sum_{\p\in \cW(j)}N_\fp \leq B_\delta\delta h_f(a) 3^n + C_\delta
 \leq -\frac{\delta\log\delta}{\log 3} h_f(a)3^n + C_\delta,
\end{align}
Adding (\ref{one}) and (\ref{two}), and observing that $\delta\log\delta\to 0$ as $\delta\to 0$, 
we are done.
\end{proof}

\section{The case of function fields}\label{function-case}

Throughout this section, $K$ will denote a function field of
transcendence degree 1 over an algebraic extension of $\Q$.  We derive
the following from work of Ghioca and Ye \cite{GY} (see also
\cite{Gauthier}) along
with Call-Silverman specialization \cite{CallSilverman}.  
 
\begin{prop}\label{FunctionFinite}
  Suppose that $f(x) = x^3 - 3a^2x + b\in K[x]$ is not isotrivial. 
  Suppose furthermore that
  $b \not= 0$ and that there are no $i,j > 0$ such that
  $f^i(a) = f^j(-a)$.  Let $c, d \in K$ be such that
  $f^\ell(a) \not=c $ and $f^\ell(-a) \not = d$ for all positive integers
  $\ell$.  Then there are at most finitely many places $v$ of $K$ such
  that there are positive integers $m,n$ with the property that
\[ \min(v(f^m(a) - c), v(f^n(-a) - d)) > 0. \]
\end{prop}

Recall that a function field $K$ of transcendence degree 1 over an algebraic extension of $\Q$
gives rise to a curve $C$ defined over $\Qbar$, and that the places of $K$
correspond to closed points in $C(\Qbar)$ (when we refer to points of $C(\Qbar)$, 
we will mean closed points with respect to the field of constants of $K$). For any element of $c \in K$ and
any point $\lambda$ of $C(\Qbar)$ such that $c$ does not have a pole at
$\lambda$, we let $c_\lambda$ denote the specialization of $c$ to
$\Qbar$ at $\lambda$ (see \cite{CallSilverman} for more details);
likewise for a rational function $\varphi \in K(x)$, we let
$\varphi_\lambda$ denote the specialization of $\varphi$ to $\Qbar(x)$
at $\lambda$ for any $\lambda$ such that the coefficients of $\varphi$
do not have poles at $\lambda$.  With notation as above, Ghioca and Ye
\cite{GY} proved the following (see also \cite{Gauthier} for a similar
result).  

\begin{thm}\label{GYT}(\cite{GY}).   Let $f(x) = x^3 - 3a^2x + b \in K[x]$ be
  nonisotrivial.  
If there is an infinite sequence
  $(\lambda_i)_{i=1}^\infty$ in $C(\Qbar)$ such that 
\[ \lim_{i \to \infty} h_{f_{\lambda_i}} (a_{\lambda_i}) =  \lim_{i \to \infty}
  h_{f_{\lambda_i}} (-a_{\lambda_i}) = 0, \]
then at least one of the following holds:
\begin{itemize}
\item $b=0$;
\item $a$ is preperiodic under $f$;
\item $-a$ is preperiodic under $f$; or
\item there are $i,j > 0$ such that $f^i(a) = f^j(-a)$.  
\end{itemize}
\end{thm}

We now prove a simple lemma that follows from work of Call and
Silverman \cite{CallSilverman}.  Note that this lemma works for any
rational function of degree greater than one over a function field,
not just for cubic polynomials.

\begin{lem}\label{FromCS}
  Let $\varphi \in K(x)$ have degree $d \geq 2$.  Let $y, z \in K$ with
  $h_\varphi(y) > 0$.  Let $(\lambda_i)_{i=1}^\infty$ be a sequence of points of 
  $C(\Qbar)$ satisfying
  $\varphi_{\lambda_i}^{n_i}(y_{\lambda_i}) = z_{\lambda_i}$ for a
  sequence $(n_i)_{i=1}^\infty$ of positive integers with
  $\lim_{i\to\infty} n_i = \infty$. Then
\begin{equation*}\label{to-zero}
\lim_{i \to \infty} h_{\varphi_{\lambda_i}}(y_{\lambda_i}) = 0 .
\end{equation*}
\end{lem}
\begin{proof}
  For each $i$ we have
\begin{equation}\label{di} h_{\varphi_{\lambda_i}} (y_{\lambda_i}) = \frac{1}{d^{n_i}}
  h_{\varphi_{\lambda_i}} (z_{\lambda_i}) 
\end{equation}
since $h_{\varphi_{\lambda_i}}(\varphi_{\lambda_i}(x)) = d h_{\varphi_{\lambda_i}}(x)$ for
all $x \in \Qbar$.  
Let $h_C$ be a height function corresponding to a divisor of degree
  one on $C$.  Then, by \cite[Theorem
  4.1]{CallSilverman}, we have
\[ \lim_{h_C(t) \to \infty} \frac{h_{\varphi_t}(y_{t})}{h_C(t)} = h_\varphi(y) \]
and
\[ \lim_{h_C(t) \to \infty}
  \frac{h_{\varphi_t}(z_{t})}{h_C(t)} = h_\varphi(z).\] Hence,
the $h_C(\lambda_i)$ must be bounded for all $i$, since otherwise \eqref{di} would
imply that $h_\varphi(y) = 0$.  Now, Theorem 3.1 of \cite{CallSilverman}
states that
\[|h_{\varphi_t}(z_{\lambda_i}) - h_{\varphi}(z)| \leq O(1)(h_C(t) +
  1)\] 
for all $t$ in $C(\Qbar)$. As $h_C(\lambda_i)$ is bounded, we see that $h_{\varphi_{\lambda_i}}(z_{\lambda_i})$ must be
bounded for all $i$.  The Lemma then follows immediately
from \eqref{di}.
\end{proof}

Now we can prove Proposition \ref{FunctionFinite}.

\begin{proof}[Proof of Proposition \ref{FunctionFinite}]
If $a$ is preperiodic, then there are only finitely many places $v$ of $K$ such
that there is an $m$ for which $v(f^m(a) - c) > 0$.  Likewise, if $-a$
is preperiodic then there are only finitely many places $v$ such
that there is an $n$ for which $v(f^n(a) - c) > 0$.  Hence, we may
assume that neither $a$ nor $-a$ is preperiodic.  We now argue by
contradiction.  Suppose there is an infinite sequence of places
$(v_i)_{i=1}^\infty$ such that for each $i$, we have $m_i, n_i$ with
the property that
\[ \min(v_i(f^{m_i}(a) - c), v_i(f^{n_i}(-a) - d)) > 0. \] Let $\lambda_i$ be the
point in $C(\Qbar)$ corresponding to $v_i$.  Then we have
$f_{\lambda_i}^{m_i}(a_{\lambda_i}) = c_{\lambda_i}$ and
$f_{\lambda_i}^{n_i}(-a_{\lambda_i}) = d_{\lambda_i}$.  Since $c$ is
not in the forward orbit of $a$ and $d$ is not in the forward orbit of
$-a$, we see that $\lim_{i \to \infty} m_i = \lim_{i \to \infty} n_i =
\infty$ since any $m_i$ or $n_i$ can only arise for a finite number of
places $v_i$.  Thus, by Lemma \ref{FromCS}, we have 
\[ \lim_{i \to \infty} h_{f_{\lambda_i}} (a_{\lambda_i}) =  \lim_{i \to \infty}
  h_{f_{\lambda_i}} (-a_{\lambda_i}) = 0. \]
Theorem \ref{GYT} then gives a contradiction since $b \not= 0$ and
there are no $i,j > 0$ such that $f^i(a) = f^j(-a)$.  
\end{proof}

\begin{rem}\label{constant-field}
  Although Ghioca and Ye only prove Theorem \ref{GYT} over $\Qbar$,
  their proof works for the algebraic closure of any finitely
  generated extension of $\bQ$.  Any finitely generated extension $M$
  of $\bQ$ can be endowed with a set of places that give rise to
  appropriate height functions on the algebraic closure of $M$ (see
  \cite[Chapter 1]{BG}, for example), and work of Gubler \cite{Gubler}
  extends Yuan's equidistribution theorem \cite{Yuan} for small points over
  $\Qbar$ to this context.  The authors of both \cite{GY} and
  \cite{Gauthier} have confirmed that their proofs can extended with
  essentially no modification to prove that Theorem \ref{GYT} holds
  when $K$ is a function field over any the algebraic closure of any
  finitely generated extension $M$ of $\bQ$ and
  $h_{f_{\lambda_i}} (\pm a_{\lambda_i})$ are canonical heights
  associated to a Weil height for $M$.  Hence, Proposition
  \ref{FunctionFinite} holds over a function field over any field of
  constants of characteristic 0, and the main results of this paper
  hold in that context as well, since nothing else in our arguments
  require that the field of constants of $K$ be $\Qbar$.
\end{rem}

\section{Proof of Main Theorem}\label{main proof}

Let $K$ be a number field or a function field of characteristic $0$ of transcendence degree 1 over an algebraic extension of $\Q$. 
If $K$ is a number field, for the rest of the section we will assume both the $abc$ conjecture for $K$ 
and Vojta's conjecture for blowups of $\P^1 \times\P^1$. 

The proof of Theorem~\ref{thm: main} combines the preliminary arguments from 
throughout the paper with the following proposition, which 
produces primes with certain ramification behavior in $K_n(\beta)$. 
Recall the definitions of Condition R and Condition U from Section \ref{Galois}.
\begin{prop}\label{final prop}
  Let $f(x) = x^3 - 3a^2x + b\in K[x]$. Assume that $a\neq 0$ and that $f$ is not PCF. 
  Let $L$ be a finite extension of $K$.  Let $\alpha_1, \dots, \alpha_s$ be distinct elements of $L$ such that
  $\alpha_i\notin\O_f(a)\cup\O_f(-a)$ and $\alpha_i\notin\O_f(\alpha_j)$ for all $i,j\in\{1,\dots,s\}$. 
  %$f^\ell(\pm a) \not= \alpha_i$ for any $\ell > 0$ and $1 \leq i \leq s$ and $f^\ell(\alpha_i) \not= \alpha_j$ for all $\ell \geq 0$ and any$1 \leq i,j \leq s$ (note this implies $\alpha_1,\dots,\alpha_s$ are distinct). 
  If $b=0$, further assume that $\alpha_i\notin\O_f(-\alpha_j)$ for all $i,j\in\{1,\dots,s\}$. 
  Then, for all sufficiently large $n$, there exist primes $\fp_1,\dots,\fp_s$
   of $L$ such that
\begin{itemize}
\item[(a)] $\fp_i \cap K(\alpha_i)$ satisfies Condition R at
$\alpha_i$ for $n$;
\item[(b)] $\fp_i \cap K(\alpha_j)$ satisfies Condition U at
$\alpha_j$ for $n$ for all $j \not= i$;
\item[(c)] $\fp_i \cap K(\alpha_i)$ does not ramify in $L$.   
\end{itemize}
\end{prop}
\begin{proof}
For any $1\leq i\leq s$, let $\mathcal{A}_i(n)$ be the set of primes 
$\fp$ of $L$ such that (a), (b), and (c) hold. If a prime $\fp$ 
of $L$ satisfies condition R or condition U at $\alpha_i$ for $n$, then it is easy 
to see that the prime $\fp\cap K(\alpha_i)$ of $K(\alpha_i)$ also satisfies condition R or 
condition U at $\alpha_i$ for $n$. Therefore we will establish Conditions R and U for primes of $L$ rather than 
primes of the various $K(\alpha_i)$, which will make the argument less cumbersome to state. 
Thus all sums below are indexed by primes of $\fo_L$ as in Section~\ref{heights}.

There are only finitely many primes $\fp$ of $\fo_L$ for which $f$ does not have good separable reduction at $\p$, $v_\fp(\alpha_i)\neq 0$ for some $i$, or $\fp\cap K(\alpha_i)$ ramifies in $L$ for some $\alpha_i$. The contributions of these primes to our estimates will be absorbed into the constant term $C_\delta'$ at the end of the proof.

Choose $i$ and $j$ with $1\leq i,j\leq s$ (and possibly $i=j$). By Lemma \ref{from-Roth}, for any $\delta>0$ there is a 
constant $C_\delta$ such that 
\begin{equation}\label{starting point}
\sum_{v_\fp(f^n(a)-\alpha_i)=1} N_\fp\geq (3-\delta)3^{n-1}h_f(a) + C_\delta.
\end{equation}
Let $\cX(n)$ be the set of primes $\fp$ with $\min(v_\fp(f^n(a)-\alpha_j)>0, v_\fp(f^m(a)-\alpha_i))>0$ for some $1\leq m\leq n-1$. By Lemma \ref{from-5.1} with $\gamma=a$, $\beta_1=\alpha_i$, and $\beta_2=\alpha_j$, we have
\begin{equation}\label{a earlier}
\sum_{\fp\in\cX(n)} N_\fp \leq \delta 3^nh_f(a) + O_\delta(1),
\end{equation}
 Furthermore, if $j\neq i$, then the set of primes $\fp$ such that $\min(v_\fp(f^n(a)-\alpha_i),v_\fp(f^n(a)-\alpha_j))>0$ is a finite set depending only on $\alpha_i$ and $\alpha_j$ (and not on $n$), because $\alpha_i\equiv\alpha_j\pmod{\fp}$ for such $\fp$. So
\begin{equation}\label{a same}
\sum_{\min(v_\fp(f^n(a)-\alpha_i),v_\fp(f^n(a)-\alpha_j))>0} N_\fp = O_\delta(1).
\end{equation}
Let $\cY(n)$ be the set of $\fp$ with $\min(v_\fp(f^n(a)-\alpha_i), v_\fp(f^m(-a)-\alpha_j))>0$ for 
some $1\leq m\leq n$. 
If $b=0$, then applying Lemma~\ref{odd} with $\beta_1=\alpha_i$ and $\beta_2=\alpha_j$ gives 
\begin{equation}\label{-a number field}
\sum_{\fp\in \cY(n)}N_\fp\leq \delta 3^n h_f(a) + O_\delta(1).
\end{equation}
If $b\neq 0$ and $K$ is a number field, then applying Lemma \ref{both critical points} with 
$\beta_1=\alpha_i$ and $\beta_2=\alpha_j$, again we arrive at \eqref{-a number field}. 
If $b\neq 0$ and $K$ is a function field, then 
instead of Lemma \ref{both critical points}, we invoke Proposition \ref{FunctionFinite} 
with $c=\alpha_i$ and $d=\alpha_j$. This gives us the stronger statement that there 
are only finitely many primes $\fp$ of $K$ such that both $v_\fp(f^n(a)-\alpha_i)>0$ 
and $v_\fp(f^m(-a)-\alpha_j)>0$ for any $1\leq m\leq n$. So for function fields, we actually have
\begin{equation}\label{-a function field}
\sum_{\fp\in \cY(n)}N_\fp = O_\delta(1),
\end{equation}
though we do not require the full strength of this bound.

Starting with \eqref{starting point} for some choice of $1\leq i\leq s$, for each $j$ with $1\leq j\leq s$ we subtract \eqref{a earlier} and either \eqref{-a number field} or \eqref{-a function field}, and for all $j\neq i$ we subtract \eqref{a same}. This sieves out all $\fp\notin\mathcal{A}_i(n)$ from the sum in \eqref{starting point}, and we have
\begin{equation}
\sum_{\fp\in\mathcal{A}_i(n)} N_\fp\geq 3^nh_f(a)(1-\delta/3-2s\delta) +C_\delta'
\end{equation}
where $C_\delta'$ is a constant obtained by combining all the $O_\delta(1)$ terms. For sufficiently large $n$, we can choose some $\delta$ to make the RHS positive. Repeating this process for each $i$, we are done.
\end{proof}

We are now ready to prove Theorem~\ref{thm: main}. Let $f\in K[x]$ be a degree 3 polynomial and let $\beta\in K$. Assume that $f$ is not postcritically finite, that $\beta$ is not postcritical for $f$, that $(f,\beta)$ is eventually stable, and that $f$ has distinct finite critical points $\gamma_1,\gamma_2$ with $f^n(\gamma_1)\neq f^n(\gamma_2)$ for any $n\geq 1$. By Proposition~\ref{reductions}, we can replace $K$ with a finite extension and replace $f$ with 
a change of variables so that $f(x)=x^3-3a^2x + b$ for $a,b\in K$. As we assumed the critical points of $f$ are distinct, we have $a\neq 0$. Without loss of generality, assume $h_f(a)\geq h_f(-a)$, and note this implies $h_f(a)>0$.
\begin{proof}[Proof of Theorem~\ref{thm: main}]
Since $f$ is eventually stable, by Proposition~\ref{eventual stability facts} there exists some $N$ such that for any $\alpha\in f^{-N}(\beta)$, the pair $(f,\alpha)$ is stable. Let $\alpha_1,\dots,\alpha_{3^N}$ be the elements of $f^{-N}(\beta)$. Note that $|f^{-N}(\beta)|=3^N$ because we have assumed that $\beta$ is not postcritical for $f$, so $f^N(x)-\beta$ has distinct roots in $\overline{K}$ for all $N\geq 1$. 

We argue that $\alpha_1,\dots,\alpha_{3^N}$ satisfy the assumptions of Proposition~\ref{final prop}.
As $\beta$ is not postcritical, no $\alpha_i$ lies in $\O_f(a)$ or $\O_f(-a)$, and all $\alpha_i$ are distinct. 
If $f^\ell(\alpha_i)= \alpha_j$ holds for $i\neq j$ and $\ell\geq 1$, then $f^n(\alpha_i)=f^{n-\ell}(\alpha_j)=\beta$ for some $m<n$ and $f^\ell(\beta)=\beta$, contradicting eventual stability of $(f,\beta)$ by Proposition~\ref{eventual stability facts}. Likewise, if $b=0$ and $f^\ell(-\alpha_i)= \alpha_j$, then $-f^\ell(\alpha_i)=\alpha_j$ because $f$ is odd, whence $f^n(\alpha_j)=f^{n-\ell}(-f^\ell(\alpha_i))=-f^n(\alpha_i)$, so $\beta=-\beta$ and $\beta=0$. Then $f(\beta)=\beta$, again contradicting eventual stability. Thus Proposition~\ref{final prop} holds. For all sufficiently large $n$,
there exist primes $\fp_1,\dots,\fp_{3^N}$ of $K_N(\beta)$ satisfying
conditions (a) through (c) of Proposition \ref{final Galois} for $\ba
= (\alpha_1, \dots, \alpha_{3^N})$.  Since $f^n(x) - \alpha_i$ is
irreducible over $K(\alpha_i)$, condition (d) holds as well. 
Thus, Proposition \ref{final
  Galois} implies that $|\Gal(K_n(\ba)/ K_{n-1}(\ba))| =6^{3^{N+n-1}}$.
Since $K_n(\ba) = K_{N+n}(\beta)$ and $K_{n-1}(\ba) =
K_{N+n-1}(\beta)$, we thus have 
\[\frac{|G_{N+n}(\beta))|}{|G_{N+n-1}(\beta))}  =    \Gal(K_{N+n}(\beta)/
K_{N+n-1}(\beta))| = \frac{|\Aut(T_{N+n})|}{|\Aut(T_{N+ n-1})|} \] 
for all
sufficiently large $n$, which means that the index of $G_\infty(f,
\beta)$ in $\Aut(T_\infty)$ must be finite, as desired.
\end{proof}

\section{The stunted tree}\label{The stunted tree}

Let $f\in K(x)$ and $\beta\in\P^1(\Kbar)$. We define the \emph{stunted tree} 
$T^{s}_n(\beta)$ inductively as follows in 
terms of its $n$th \emph{level} $\cL_n(\beta)$. For $n=0$, set 
$T^s_0(\beta)=\cL_0(\beta)=\{\beta\}$. For $n\geq 1$, define $\cL_n(\beta)$ by
\[
\cL_n(\beta)=\{z\in\P^1(\Kbar):f^n(z)=\beta\}\setminus T^{s}_{n-1}(\beta).
\]
In other words, $\cL_n(\beta)$ consists of the ``strict'' inverse images
of $\beta$ under $f^n$, the $z \in f^{-n}(\beta)$ that did not appear
as inverse images of $\beta$ under $f^m$ for some $m < n$.  
Then define
\[
T^{s}_n(\beta)=\bigcup_{i=0}^n \cL_i(\beta),
\]
noting that this union is disjoint by definition. The idea behind this construction is that $T^{s}_n(\beta)$ is the ``actual" tree of preimages of $\beta$ as 
opposed to the ``idealized" tree of preimages $T_n(\beta)$. Each point in $T^{s}_n(\beta)$ occurs precisely once, 
as points are not repeated in the tree due to periodicity or 
to account for ramification at critical points. In particular, $\cL_n(\beta)$ is 
a set and not a multiset. For example, see Figure~\ref{T^s_2 examples} and compare with 
Figure~\ref{T_2 examples} from Section~\ref{Background}.

\begin{figure}[ht]
\centering
\subfloat{\includegraphics[scale=.65]{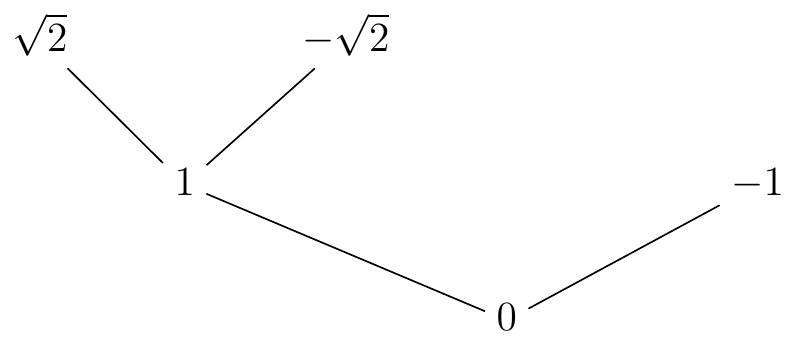}}\hspace{.75in}
\subfloat{\includegraphics[scale=.65]{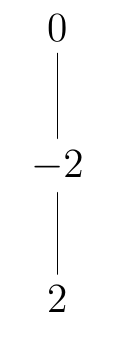}}
\caption{$T^s_2(0)$ for $x^2-1$ and $T^s_2(2)$ for $x^2-2$}
\label{T^s_2 examples}
\end{figure}

Let $T^{s}_\infty(\beta)$ be the direct limit of the $T^{s}_n(\beta)$. 
As before, the group $G_n(\beta)$ acts faithfully on the tree $T^{s}_n(\beta)$, and this action 
commutes with the action of $f$ on $T^{s}_n(\beta)\setminus\beta$, so 
there is an injection $G_n(\beta)\hookrightarrow \Aut(T^{s}_n(\beta))$ for each $n$. Taking 
inverse limits, there is an injection $G_\infty(\beta)\hookrightarrow\Aut(T^{s}_\infty(\beta))$.

For $n\in \N\cup\{\infty\}$, there is a surjective morphism of rooted
trees $T_n(\beta)\to T^{s}_n(\beta)$ where the fiber over each
$z\in T^{s}_n(\beta)$ is the set of vertices in $T_n(\beta)$ labeled
by $z$.  Thus $T^s_n(\beta)$ is a quotient tree of $T_n(\beta)$.  The
quotient morphism is an isomorphism if and only if no element of
$\P^1(\Kbar)$ occurs more than once in $T_\infty$; this holds in turn
if and only if $\beta$ is not periodic and not postcritical for $f$.
Also note that the morphism is Galois-equivariant. So there are
injections $\Aut(T^{s}_n(\beta))\hookrightarrow \Aut(T_n(\beta))$ for
each $n\in\N\cup\{\infty\}$, and the image of the arboreal Galois
representation $\rho_{f,\beta}:\Gal(\Kbar/K)\to\Aut(T_\infty)$ is
contained in the subgroup $\Aut(T^{s}_\infty(\beta))$.  If
$[\Aut(T_\infty):\Aut(T^{s}_\infty(\beta))]=\infty$, then
$G_\infty(\beta)$ cannot have finite index in $\Aut(T_\infty)$, but it
may still be possible that $G_\infty(\beta)$ has finite index in
$\Aut(T^s_\infty(\beta))$. Thus we have a variant of our original
finite index problem, which does not appear to have been studied
before.  
\begin{question}\label{stunted finite index question}
When is $[\Aut(T^{s}_\infty(\beta)):G_\infty(\beta)]<\infty$?
\end{question}
If $\beta$ is periodic or postcritical, then Question~\ref{stunted
  finite index question} may have a positive answer even though
Question~\ref{main question} cannot.

We define a variant of eventual stability that can hold even when $\beta$ is periodic 
or postcritical for $f$.
\begin{defn}
We say $(f,\beta)$ is \emph{tree-stable} if there exists $n$ such that, for each $\alpha\in \cL_n(\beta)$, the 
pair $(f,\alpha)$ is stable.
\end{defn}
If the pair $(f,\beta)$ is eventually stable, then it is also tree-stable by 
Proposition~\ref{eventual stability facts}. Of course, 
the converse fails in general.

In Theorem~\ref{finite index in stunted tree} we give a nearly complete answer to 
Question~\ref{stunted finite index question} for cubic polynomials. We prove a set of necessary and sufficient 
conditions for $G_\infty(\beta)$ to have finite index in $\Aut(T^{s}_\infty(\beta))$ when $K$ is a function field, 
provided that we exclude the possibility that $f$ is odd and $\beta=0$. 
If $K$ is a number field, 
the proof is conditional on $abc$ and Vojta's conjecture just as in Theorem~\ref{thm: main}.

The basic strategy to prove a finite index result in $\Aut(T^s_\infty(\beta))$ is the same as for $\Aut(T_\infty)$. 
That is, 
we show that $\Gal(K_n(\beta)/K_{n-1}(\beta))$ is maximal for all large $n$, which in this case means
\[
\Gal(K_n(\beta)/K_{n-1}(\beta))\cong \Aut(T^s_{n}(\beta)/T^s_{n-1}(\beta)),
\]
where $\Aut(T^s_{n}(\beta)/T^s_{n-1}(\beta))$ is the set of all automorphisms of 
$T^s_n(\beta)$ that act trivially on $T^s_{n-1}(\beta)$. Then finite index in $\Aut(T^s_\infty(\beta))$ 
follows easily from the profinite structure of $G_\infty(\beta)$.

There is one more issue that arises in this context. If $f$ is odd
and there is a $\gamma$ such that both $\gamma$ and $-\gamma$ are in
$T^s_\infty(\beta)$, then $-\beta \in T^s_\infty(\beta)$ so there is
some $\ell$ such that $f^\ell(-\beta) = \beta$.  This means that
$f^{2 \ell}(\beta) = - f^\ell(\beta) = \beta$ so $\beta$ is in a
periodic cycle that is mapped to itself by $\sigma: x \mapsto -x$.
Hence the entire tree $ T^s_\infty(\beta)$ must be stable under
$\sigma: x \mapsto - x$.  This forces
$[\Aut(T_\infty^s(\beta): G_\infty(\beta)] = \infty$ (see Proposition
\ref{partial} for details in a more general context).  This is similar to the
dynamical analogs of complex multiplication described by Jones~\cite[Section
3.4]{RafeArborealSurvey}.

\begin{thm}\label{finite index in stunted tree}
  Let $K$ be a number field or a function field of characteristic $0$
  and transcendence degree one over an algebraic extension of $\Q$.
  Let $f(x) = x^3 - 3a^2x + b \in K[x]$ and let $\beta\in K$.  If $K$ is a
  number field, assume the $abc$ conjecture for $K$ and Vojta's
  conjecture for blowups of $\P^1 \times\P^1$. If $K$ is a function
  field, assume that $f$ is not isotrivial.  If $f$ is odd, suppose
  furthermore that there is no $\gamma$ such that both $\gamma$ and
  $-\gamma$ are in $T^s_\infty(\beta)$. 
  
  The following are equivalent:
  \begin{enumerate}
  \item The pair $(f,\beta)$ is tree-stable, $f$ is not PCF, $f$ has
    distinct finite critical points $a$, $-a$, and we have
    $f^n(a)\neq f^n(-a)$ for all $n\geq 1$.
   \item The group $G_\infty(\beta)$ has finite index in $\Aut(T^{s}_\infty(\beta))$.
  \end{enumerate}
\end{thm}

\begin{proof}
Assume that $\beta$ is either periodic or postcritical for $f$, as otherwise 
the theorem follows immediately from Theorem~\ref{thm: main}.

The necessity of the the first three conditions in (1) is follows from
their necessity in Theorem~\ref{thm: main}.  In particular,
$f^{-1}(\beta) \not= \{ \beta \}$ (since if it was then $f(x) = x^3$,
which is post-critically finite), so we may choose a $\gamma$ that is
not periodic or post-critical such that $f^n(\gamma) = \beta$ for some
$n$.  If $G_\infty(\beta)$ has finite index in
$\Aut(T^{s}_\infty(\beta))$, then certainly $G_\infty(\gamma)$ has
finite index in $\Aut(T_\infty(\gamma))$, but as seen in Proposition
\ref{reductions}, that can only happen if $f$ is not PCF, $f$ has
distinct finite critical points $\gamma_1,\gamma_2$, and
$f^n(\gamma_1)\neq f^n(\gamma_2)$ for all $n\geq 1$.

We turn to the problem of showing the listed conditions are
sufficient.  We may choose an $N$ such that none of the elements in
$\cL_N(\beta)$ are periodic are post-critical or periodic and such
that $f^n(x) - \gamma$ is irreducible for all
$\gamma \in \cL_N(\beta)$.

Let $\{ \alpha_1, \dots, \alpha_t \}$ denote the distinct elements of
$\cL_N(\beta)$.  Then $f^N(\alpha_i) = \beta$ for all $\alpha_i$ and
$f^m(\alpha_i) \not= \beta$ for all $m < N$.  If
$f^\ell(\alpha_i) = \alpha_j$ for $\ell > 0$, then
$f^{N-\ell}(\alpha_j) = \beta$, a contradiction.  Thus
$\alpha_i \notin \O_f(\alpha_j)$ for $i \not = j$.  Since no $\alpha_i$
is in $\O_f(a)$ or $\O_f(-a)$, we thus see that if $b \not= 0$, then
$\alpha_1, \dots, \alpha_t$ satisfy the conditions of Proposition
\ref{final prop}.  If $b = 0$, we
see that if $f^\ell(-\alpha_j) = \alpha_i$ for some $\ell > 0$ then
$\alpha_j$ and $-\alpha_j$ are both in $T^s_\infty(\beta)$, a
contradiction.  Thus, the conditions of Proposition
\ref{final prop} are also met in this case. 
 
Letting $\ba = (\alpha_1, \dots,
\alpha_t)$ and applying Proposition \ref{final Galois}, as in the
proof of Theorem \ref{thm: main}, we then see that for all
sufficiently large $n$, we have
\[ \frac{|G_{N+n}(\beta)|}{|G_{N+n-1}(\beta)|} =
    \frac{|\Aut(T^s_{N+n}(\beta))|}{|\Aut(T^s_{N+n-1}(\beta))|}, \]
and our proof is complete.  

\end{proof}

\section{The multitree}\label{The multitree}

Let $f\in K(x)$ with $\deg f\geq 2$. Throughout the paper we have studied the Galois theory of the 
fields generated by taking preimages of one point $\beta$ under $f$. In this section, 
we simultaneously consider a collection of $s$ (not necessarily disjoint) trees rooted at distinct points  
$\beta_1,\dots,\beta_s\in\P^1(\Kbar)$. To ease notation, 
set ${\bf B} = \{\beta_1,\dots,\beta_s\}$. 
Recall the definition of $\cL_n(\beta_j)$ from Section~\ref{The stunted tree} 
as the $n$th level of a stunted tree rooted at $\beta_j$. 
Define
\[
\cM_n({\bf B})= \bigcup_{i=0}^n \bigcup_{j=1}^s \cL_n(\beta_j)
\]
and
\[
G_n({\bf B}) = \Gal(K(\cM_n({\bf B} ))/K({\bf B} )).
\]
We refer to $\cM_n({\bf B})$ as a \emph{multitree}. See Figure~\ref{multitree example} for an example.
\begin{figure}[ht]
\centering
\subfloat{\includegraphics[scale=.65]{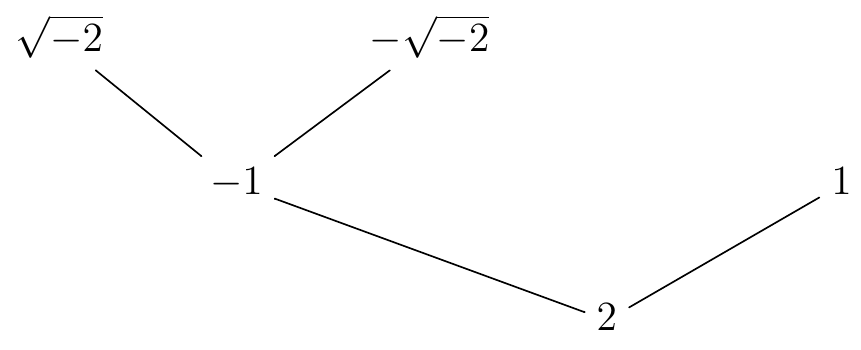}}\hspace{.5in}
\subfloat{\includegraphics[scale=.65]{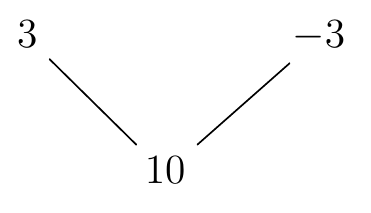}}
\caption{$\cM_1(-1,2,10)$ for $x^2+1$}
\label{multitree example}
\end{figure}
As usual, define $\cM_\infty({\bf B})$ to be 
the direct limit of the $\cM_n({\bf B})$ and $G_\infty({\bf B})$ to be the inverse limit 
of the $G_n({\bf B})$ as $n\to\infty$. For each $n$, $G_n({\bf B}) $ acts faithfully on $\cM_n({\bf B})$ in the usual way. So there are injections 
$G_n({\bf B}) \hookrightarrow\Aut(\cM_n({\bf B}))$, and thus an injection 
$G_\infty({\bf B}) \hookrightarrow\Aut(\cM_\infty({\bf B}))$, where an automorphism of the 
multitree must fix each root. Once again we have a finite index question: 
\begin{question}
When is $[\Aut(\cM_\infty({\bf B})):G_\infty({\bf B})]<\infty$? 
\end{question}
We partially answer this question for cubic polynomials in Theorem~\ref{multitree theorem}, 
which is a multitree version of Theorem~\ref{finite index in stunted tree}. As in 
Theorem~\ref{finite index in stunted tree}, we exclude certain cases arising from
odd polynomials (and our proof is conditional on $abc$ and Vojta's conjecture when $K$ is a number field). 
First we make some remarks about the structure of $\Aut(\cM_\infty({\bf B}))$.

Each stunted tree $T^s_\infty(\beta_i)$ injects into $\cM_\infty({\bf B})$, and the root of $T^s_\infty(\beta_i)$ 
maps to a root of $\cM_\infty({\bf B})$. 
If no $\beta_i$ is in the forward orbit of any other $\beta_j$, then $\cM_\infty$ is precisely the disjoint 
union of the rooted trees $T^s_\infty(\beta_i)$ for $1\leq i\leq s$ (this is known as a \emph{forest}). 
In general, there is a root-preserving surjective morphism of graphs
\[
\bigsqcup_{i=1}^s T^s_\infty(\beta_i)\to \cM_\infty({\bf B})
\]
where the fiber over each $z\in \cM_\infty({\bf B})$ is the set of all vertices in the various trees labeled 
by $z$. This morphism is an isomorphism if and only if the trees $T^s_\infty(\beta_i)$ are disjoint. 
If the trees intersect, then we can describe $\cM_\infty({\bf B})$ as follows: 
partition ${\bf B}$ into classes based on membership in the same grand orbit, where 
grand orbits of $f$ in $\Kbar$ are equivalence classes of the relation $\sim$ defined by 
$y\sim z$ if there exist $n,m\geq 0$ such that $f^n(y)=f^m(z)$. From each class, 
choose an element $\beta$ that is ``farthest forward" 
in its grand orbit, that is, such that either $\beta$ is periodic or 
$f^n(\beta)\notin{\bf B}$ for all $n>0$. 
Let ${\bf A} = \{\gamma_1,\dots,\gamma_{w}\}$ be the subset of ${\bf B}$ 
consisting of the chosen class representatives.
Now we have the equality of unrooted graphs
\[
\cM_\infty({\bf B}) = \cM_\infty({\bf A}) = \bigsqcup_{i=1}^w T^s_\infty(\gamma_i),
\]
but the points $\beta_i$ are also marked as roots within each tree $T^s_\infty(\gamma_i)$, and are not moved
by any element of $\Aut(\cM_\infty({\bf B}))$. 

We describe the automorphism group of this disjoint union of trees. Let 
$A_i$ be the subgroup of $\Aut(T^s_\infty(\gamma_i))$ that acts as the identity on the elements of ${\bf B}$ that 
lie in $T^s_\infty(\gamma_i)$. It is easy to see that $A_i$ has finite index in $\Aut(T^s_\infty(\gamma_i))$.  
The group $A_i$ sits as a normal subgroup in 
$\Aut\left(\bigsqcup_{i=1}^w T^s_\infty(\gamma_i)\right)$, and any 
$g\in \Aut\left(\bigsqcup_{i=1}^w T^s_\infty(\gamma_i)\right)\setminus\prod_{i=1}^w\Aut(T^s_\infty(\gamma_i))$ must 
move some basepoint $\gamma_i$, i.e. $g(\gamma_i)\neq \gamma_i$ for some $i$. On the other hand, 
specifying a permutation of the set ${\bf A}$ and an automorphism of each stunted tree completely 
determines an automorphism of the disjoint union. Therefore
\[
\Aut(\cM_\infty({\bf B})) \cong \prod_{i=1}^w A_i\rtimes H
\]
for some subgroup $H$ of the symmetric group $\text{Sym}({\bf A}) \cong S_w$. More precisely, $H$ consists 
of the permutations in $S_w$ that preserve the partition of ${\bf A}$ defined by grouping together 
those basepoints $\gamma_i$ such that $A_i$ is isomorphic.  As in the
case of Theorem \ref{finite index in stunted tree}, there are some
added complications when $f$ is odd.

\begin{thm}\label{multitree theorem}
  Let $K$ be a number field or a function field of characteristic $0$
  and transcendence degree one over an algebraic extension of $\Q$.  Let
  $f(x) = x^3 - 3a^2 x + b \in K[x]$ with $\deg f=3$ and let $\beta_1,\dots,\beta_s \in K$ be
  distinct.  If $K$ is a number field, assume the $abc$ conjecture for
  $K$ and Vojta's conjecture for blowups of $\P^1 \times\P^1$. If $K$
  is a function field, assume that $f$ is not isotrivial. If $f$ is an
  odd polynomial, assume furthermore that there is no $\gamma$ such
  that both $\gamma$ and $-\gamma$ are in $\cM_\infty({\bf B})$. 
   
The following are equivalent:
    \begin{enumerate}
  \item Each pair $(f,\beta_i)$ is tree-stable for $1\leq i\leq s$,
   $f$ is not PCF, 
   $f$ has distinct finite critical points $a, -a$, and
   $f^n(a)\neq f^n(-a)$ for all $n\geq 1$.

   \item The group $G_\infty({\bf B})$ has finite index in $\Aut(\cM_\infty({\bf B}))$.
  \end{enumerate}
\end{thm}

\begin{proof}
  Let ${\bf A}=\{\gamma_1,\dots,\gamma_w\}$ be a set of grand orbit
  representatives for ${\bf B}$ as in the discussion before the theorem.

  If any one of the conditions fails, then by Theorem~\ref{finite
    index in stunted tree}, one of the groups $G_\infty(\beta_i)$ has
  infinite index in $\Aut(T^s_\infty(\beta_i))$. As in the proof of
  Theorem~\ref{finite index in stunted tree}, this forces
  $G_\infty(\beta_i)$ to have infinite index in
  $\Aut(T^s_\infty(\gamma_j))$ (where $\gamma_j$ is the chosen grand
  orbit representative for $\beta_i$) and thus infinite index in
  $A_j$. From the discussion of the structure of
  $\Aut(\cM_\infty({\bf B}))$ before the statement of the theorem, it
  follows easily that $G_\infty({\bf B})$ has infinite index in
  $\Aut(\cM_\infty({\bf B}))$.

  Now assume that all of the conditions hold.  Take $N$ to be large
  enough so that ${\bf B}\subseteq \cM_N({\bf A})$ and that, for
  $1\leq i\leq w$, each $\alpha \in \cL_N(\gamma_i)$ is not periodic
  and not postcritical and has the property that $f^n(x) - \alpha$ is
  irreducible over $K(\alpha)$ for all $n$. Then, in particular, for
  every positive integer $n$ and every $\alpha \in \cL_N(\gamma_i)$,
  the tree $T^s_\infty(\alpha)$ is the full $d$-ary tree
  $T_\infty(\alpha)$. Therefore
\[
K(\cM_{N+n}({\bf A}))=K(\cM_N({\bf A}))
\left(\bigsqcup_{i=1}^w\bigsqcup_{z\in \cL_n(\alpha_i)} T_n(z)\right).
\]
By our choice of $N$, we have
$K(\cM_n({\bf B}))\subseteq K(\cM_{N+n}({\bf A}))$; thus
$K(\cM_\infty({\bf A}))=K(\cM_\infty({\bf B}))$ and
$G_\infty({\bf B}) =\Gal(K(\cM_\infty({\bf A}))/K({\bf B}))$.  

We now argue exactly as in the proof of Theorem~\ref{finite index in
  stunted tree}. We write 
\[ \bigcup_{i = 1}^w \cL_N(\gamma_i) = \{ \alpha_1 , \dots, \alpha_t
  \} \] We cannot have $f^\ell(\alpha_i) = \alpha_j$ for $\ell > 0$
and $i \not= j$ since $\alpha_i, \alpha_j$ are either in different
grand orbits or are strict $N$-th inverse images of the same
$\gamma \in {\bf A}$.  If $f$ is odd, then there is no $\ell$ such
that $f^{\ell}(-\alpha_j) = \alpha_i$ since otherwise we would have
that $\alpha_j$ and $-\alpha_j$ are both in $\cM_\infty({\bf B})$.

Letting $\ba = (\alpha_1, \dots,
\alpha_t)$ and applying Proposition \ref{final Galois}, as in the
proof of Theorem \ref{thm: main}, we then see that for all
sufficiently large $n$, we have 
\[ \frac{|\Gal(K_n(\ba))|}{|\Gal(K_{n-1}(\ba))|} =
    \frac{|\Aut(\bigsqcup_{i=1}^w T^s_{N+n}(\gamma_i))|}{{|\Aut(\bigsqcup_{i=1}^w T^s_{N+n-1}(\gamma_i))|}}, \]
and our proof is complete, exactly as in the proof of Theorem \ref{finite index in
  stunted tree}.  

\end{proof}

We also have the partial converse.  

\begin{prop}\label{partial}
  With notation as above, suppose that $f(x) = x^3 - 3a^2 x$ (i.e.,
  $f$ is odd) with $a \not= 0$ and that there exists $\gamma \in \Kbar$ such that both
  $\gamma$ and $-\gamma$ are in $\cM_\infty({\bf B})$.  Then
  $[\Aut(\cM_\infty({\bf B})):G_\infty({\bf B}) )]= \infty$.
\end{prop}
\begin{proof}
We may choose $\alpha, -\alpha$ such that $f^n(\alpha) = \gamma$ and
neither $\alpha$ nor $-\alpha$ is periodic or post-critical (note that
this implies in particular that $\alpha \not= 0$), Since
$\Gal(\Kbar/K(\alpha))$ has finite index in $\Gal(\Kbar/K)$, we may
suppose that $\alpha, -\alpha \in K$.  Then the trees
$T_\infty(\alpha)$ and $T_\infty(-\alpha)$ are both stable under the
action of $\Gal(\Kbar/K)$, and the map $\sigma: x \mapsto -x$
transposes $T_\infty(\alpha)$ and $T_\infty(-\alpha)$.  Hence, any
element of  $\Gal(\Kbar/K)$ that acts trivially on $T_\infty(\alpha)$
must act trivially on  $T_\infty(-\alpha)$ as well.  This means that
the image of $\Gal(\Kbar/K)$ in $\Aut(T_\infty(\alpha)) \times
\Aut(T_\infty(-\alpha))$ (induced by the action of $\Gal(\Kbar/K)$ on
$\cM_\infty({\bf B})$) cannot have finite index in $\Aut(T_\infty(\alpha)) \times
\Aut(T_\infty(-\alpha))$.  Since
$\Aut(\cM_\infty({\bf B}))$ contains $\Aut (T_\infty(\alpha)) \times
\Aut (T_\infty(-\alpha))$ as a subgroup, we must therefore have
$[\Aut(\cM_\infty({\bf B})):G_\infty({\bf B}) )]= \infty$, as
desired.  

\end{proof}

\section{The isotrivial case} \label{iso case}

In this section, we treat the case where $f$ is isotrivial.  Recall
that we say that a rational function $\phi$ over a function field $K$
with field of constants $k$ is {\em isotrivial} if there exists
$\sigma \in \Kbar(x)$ of degree one such that
$\sigma \phi \sigma^{-1} \in {\overline k}(x)$.  Note that since one
can always choose a map that sends any given point to the point at
infinity, when $f$ is a polynomial this is equivalent to saying there
is a $\sigma \in \Kbar(x)$ of degree one such that
$\sigma \phi \sigma^{-1} \in {\overline k}[x]$.  We first the case
where $f$ is a cubic with coefficients in $k$ and $\beta$ is
transcendental over $k$.

%% Add that $\beta$ is transcendental
\begin{prop}\label{trans}
  Let $k$ be any field of characteristic 0, and let $f$ be a cubic
  polynomial in $k[x]$ such that $f$ is not PCF, $f$ has distinct
  finite critical points $\gamma_1,\gamma_2$, and
  $f^n(\gamma_1)\neq f^n(\gamma_2)$ for all $n\geq 1$.  Suppose that
  $\beta$ is transcendental over $k$; let $K = k(\beta)$.  Then
  $G_\infty(f, \beta) =  \Aut(T_\infty).$
\end{prop}
\begin{proof}
  After change of variables, we may write $f(x) = x^3 - 3a^2x + b$. 
  Since there is no $n$ such that $f^n(a) = f^n(-a)$ and f is not
  post-critically finite, we may assume that $a$ is not preperiodic
  and that there are no $i,j$ with $i \leq j$ such that
  $f^i(-a) = f^j(a)$.  The polynomial $f^n(x) - \beta$ is easily seen
  to be irreducible over $k(\beta)$; for example, this follows
  immediately by looking at its Newton polygon at the place at
  infinity for $k(\beta)$.  Arguing as in \cite[Section 3]{JKMT} and
  the proof of \cite[Theorem 5.1]{BridyTucker}, we see that for any
  $n$, the prime $(\beta - f^n(a))$ in $K = k(\beta)$ satisfies
  Condition $R$ at $\beta$ for $n$.  Then, as in the proof of Theorem
  \ref{thm: main}, we have that
  $\Gal(K_n(\beta)/K_{n-1}(\beta)) \cong S_3^{3^{n-1}}$.  Hence, by
  induction, $|G_n(f, \beta)| = |\Aut(T_n)|$ for all $n$.
\end{proof}

The following is now an immediate consequence of Proposition \ref{trans}
and Lemma \ref{reductions}.
\begin{cor}\label{iso}
  Let $k$ be any algebraically closed field of characteristic 0, let
  $K$ be a function field of transcendence degree 1 over $k$, and
  let $f$ be a cubic polynomial in $K[x]$ such that $f$ is not PCF,
  $f$ has distinct finite critical points $\gamma_1,\gamma_2$, and
  $f^n(\gamma_1)\neq f^n(\gamma_2)$ for all $n\geq 1$.  Suppose that
  there exists $\sigma \in \Kbar(x)$ of degree one such that
  $\sigma f \sigma^{-1} \in k(x)$.  If $\sigma(\beta) \notin k$, then
  $G_\infty(f, \beta)$ has finite index in $\Aut(T_\infty)$.
\end{cor}

We have a much more general result in the situation where the
automorphism sending the coefficients of a rational function into the
field of constants also sends $\beta$ into the field of constants.  

\begin{thm}
Let $k$ be any algebraically closed field, and let $K$ be any function
field of transcendence degree 1 over $k$.  Suppose that $\phi$ is a
rational function in $K(x)$ such that there is a $\sigma \in \Kbar(x)$ of
degree one such that $\sigma f \sigma^{-1}  \in k(x)$.  Then, for any
$\beta$ such that $\sigma(\beta) \in k$, the group $G_\infty(\phi, \beta)$
must be finite.
\end{thm}
\begin{proof}
  We argue as in \cite[Theorem 5.1]{BridyTucker}.  For any $n$ and any
  $\alpha$ such that $f^n(\alpha) = \beta$, we have
  $\sigma f^n \sigma^{-1} \sigma (\alpha) = \sigma(\beta) \in k$, so
  $\sigma(\alpha) \in k$, since $k$ is algebraically closed. Thus,
  $\alpha = \sigma^{-1}(z)$ for some $z \in k$.  Since there
  is a finite extension $K'$ of $K$ such that $\sigma \in K'(x)$, we
  see that there is a finite extension $K'$ of $K$ such that for any
  $\alpha$ and $n$ such that $f^n(\alpha) = \beta$, we have $\alpha
  \in K'$, and our proof is complete.
\end{proof}

Since Proposition \ref{obstructions} holds for isotrivial maps, we
now have the following analog of Theorem \ref{thm: main} for isotrivial
maps.

\begin{thm}\label{iso: main}
  Let $f$ be a cubic polynomial in $K[x]$, where $K$ is a function
  field of transcendence degree 1 over an algebraically closed
  field $k$ of characteristic 0.  Let $\beta \in K$. Suppose there exists 
  $\sigma \in \Kbar[x]$ of degree one such that
  $\sigma f \sigma^{-1} \in k[x]$.

 The following are equivalent:
  \begin{enumerate}
  \item  We have $\sigma(\beta) \notin k$,
   $f$ is not PCF, 
   $f$ has distinct finite critical points $\gamma_1,\gamma_2$, and 
   $f^n(\gamma_1)\neq f^n(\gamma_2)$ for all $n\geq 1$.
   \item The group $G_\infty(f, \beta)$ has finite index in $\Aut(T_\infty)$.
  \end{enumerate}
\end{thm}

\bibliographystyle{amsalpha}
\bibliography{NewFIC}

\end{document}